\documentclass[12pt]{article}
\usepackage{amsfonts}
\usepackage[letterpaper]{geometry}
\usepackage{amsmath}
\usepackage{graphicx}
\usepackage[active]{srcltx}
\usepackage{url}
\numberwithin{equation}{section}

\usepackage{color}
\include{srctex}

\newtheorem{theorem}{Theorem}[section]

\newtheorem{lemma}[theorem]{Lemma}

\newtheorem{problem}[theorem]{Problem}
\newtheorem{proposition}[theorem]{Proposition}
\newtheorem{remark}[theorem]{Remark}

\newenvironment{proof}[1][Proof]{\textbf{#1.} }{\hfill  $\Box$}

\def\R{{\Bbb R}}
\def\E{{{\Bbb E}}}
\def\P{{\Bbb P}}

\def\Z{{\Bbb Z}}
\def\F{{\cal F}}

\def\al{{\alpha}}

\def\frac#1#2{{#1\over #2}}

\def\square{{\vcenter{\vbox{\hrule height.3pt
         \hbox{\vrule width.3pt height5pt \kern5pt
            \vrule width.3pt}
         \hrule height.3pt}}}}

\def\tlint{{- \kern-0.85em \int \kern-0.2em}}  
\def\dlint{{- \kern-1.05em \int \kern-0.4em}}  

\def\si{\sigma}

\def\cS {{\cal S}}

  \def\bR {{\Bbb R}}

\def\RR{\mathbb{R}}
\def\EE{\mathbb{E}}

\def\cA{{\cal A}}

\def\cD{{\cal D}}

\def\cF{{\cal F}}

\def\cH{{\cal H}}

\def\de{{\delta}}

\def\si{{\sigma}}

\def\al{{\alpha}}

\def\de{{\delta}}

\def\si{{\sigma}}

\def\e{{\varepsilon}}
\def \eref#1{\hbox{(\ref{#1})}}

\def\EE{\mathbb{ E}\ }
\def \eref#1{\hbox{(\ref{#1})}}

\def\si{{\sigma}}
\def\al{{\alpha}}

\makeatletter

\begin{document}

\title{Temporal  asymptotics  for  fractional parabolic Anderson model}
\author{Xia Chen, Yaozhong Hu, Jian Song, and Xiaoming Song}
\date{April 12, 2016} \maketitle
\begin{abstract}
 In this paper, we consider  fractional parabolic equation of the form $ \frac{\partial u}{\partial t}=-(-\Delta)^{\frac{\alpha}{2}}u+u\dot W(t,x)$,  where  $-(-\Delta)^{\frac{\alpha}{2}}$ with $\alpha\in(0,2]$ is a fractional Laplacian and $\dot W$ is a  Gaussian noise colored in space and time.  The  precise moment Lyapunov exponents for the Stratonovich solution and the Skorohod solution are obtained by using a variational inequality and a Feynman-Kac type large deviation result for space-time Hamiltonians driven by $\alpha$-stable process.  As a byproduct, we obtain the critical values for $\theta$ and $\eta$ such that $\E\exp\left(\theta\left(\int_0^1 \int_0^1 |r-s|^{-\beta_0}\gamma(X_r-X_s)drds\right)^\eta\right)$ is finite, where $X$ is $d$-dimensional symmetric $\alpha$-stable process and $\gamma(x)$ is $|x|^{-\beta}$ or $\prod_{j=1}^d|x_j|^{-\beta_j}$.
\end{abstract}



\section{Introduction}
Let    $\left\{\dot W(t,x)\,, t\ge 0\,, x\in \RR^d\right\}$ be  a general mean zero  
Gaussian noise on some probability space 
$(\Omega, {\cal F}, \P)$ whose  covariance function is given by 
\[
\E[\dot W(r,x) \dot W(s,y)]=|r-s|^{-\beta_0}\gamma(x-y),\
\]
where $\beta_0\in[0,1)$ and
\[
\gamma(x)=
\begin{cases}
|x|^{-\beta} &\qquad \hbox{where $\beta\in[0,d) $ \ or }\\
\prod_{j=1}^d\vert x_j\vert^{-\beta_j} &\qquad \hbox{where $\beta_j\in[0,1), j=1,\dots, d$}.\\
\end{cases}
\]
 If we abuse the notation $\beta=\sum_{i=1}^d \beta_i$,   the spatial covariance function has the following scaling property  
\begin{equation}\label{gammascaling}
\gamma(cx)=\vert c\vert^{-\beta}\gamma(x)
\end{equation} 
 for both cases. 
In this paper,  we shall study  the following fractional parabolic Anderson model,
\begin{equation}\label{pam}
\begin{cases}
\dfrac{\partial u}{\partial t}=-(-\Delta)^{\frac{\alpha}{2}}u+u\dot W(t,x), \, t>0, x\in \R^d \\
u(0,x)=u_0(x),\, x\in \R^d\,, 
\end{cases}
\end{equation}%
where  $-(-\Delta)^{\frac\al2}$ with  $0<\al\le 2$
 is the fractional Laplacian and where the initial condition satisfies  $0<\delta\le |u_0(x)|\le M<\infty$.  Without loss of of generality, we assume $u_0(x)\equiv 1$ when we study the long-term asymptotics of $u(t,x)$. The product $u\dot W(t,x)$ appearing in the above equation  will be understood in the sense of Skorohod and in the sense of Stratonovich. 
 
Let us recall some results from \cite{Song} for the SPDE \eqref{pam}.  
\begin{enumerate}
\item[(i)] 
Theorem 5.3 in \cite{Song} implies that, under the following condition:
\begin{equation}\label{Dalang}
 \beta<\alpha\,, 
\end{equation}
Eq. \eref{pam} in the Skorohod sense has a unique mild solution $\tilde u(t,x)$, and its $n$-th moment can be represented as  (see \cite[Theorem 5.6]{Song})
\begin{equation}\label{momskr}
  \E[\tilde u(t,x)^n]=\E_X\left[\prod_{j=1}^n u_0(X_t^j+x)\exp\left(\sum_{1\le j<k\le n} \int_0^t\int_0^t |r-s|^{-\beta_0}\gamma(X_r^j-X_s^k)drds\right)\right]\,,
  \end{equation}
where   $X_1,\dots, X_n$ are $n$ independent copies of   $d$-dimensional {\em symmetric  $\alpha$-stable process}   and are  independent of $W$, and   $\E_{X}$ denotes the expectation with respect to $
(X_{t}^{x}, t\ge 0)$. 
\item[(ii)]  Under a more restricted condition: 
\begin{equation}\label{Dalang'}
\alpha\beta_0+\beta<\alpha 
\end{equation}
the following Feynman-Kac formula
\begin{equation}
u(t,x)=\E_{X}\left[ u_0(X_{t}^{x})\exp \left(
\int_{0}^{t}\int_{\mathbb{R}^d} \delta_0
(X_{t-r}^{x}-y)W(dr,dy)\right) \right] \,,  \label{feynman}
\end{equation}
 is a mild Stratonovich solution to (\ref{pam})  (see   \cite[Theorem 4.6]{Song}),  where 
     $\de_0(x)$ is the Dirac delta function. Consequently, Theorem 4.8 in \cite{Song} provides a Feynman-Kac type representation for $n$-th moment of $u(t,x)$
\begin{equation}\label{momstr}
 \E[u(t,x)^n]=\E\left[\prod_{j=1}^n u_0(X_t^j+x)\exp\left(\frac12\sum_{j,k=1}^n \int_0^t\int_0^t |r-s|^{-\beta_0}\gamma(X_r^j-X_s^k)drds\right)\right]\,.
 \end{equation}
  
\end{enumerate}
 
  
The more restricted condition \eqref{Dalang'} is to ensure that the ``diagonal" terms,  i.e., 
the sum $\sum_{k=1}^n \int_0^t\int_0^t |r-s|^{-\beta_0}\gamma(X_r^k-X_s^k)drds$ appearing in \eqref{momstr} are exponentially integrable (see  Lemma \ref{integrability} and  Theorem \ref{thm-exponential}, or \cite[Theorem 3.3]{Song} in a more general setting).  To deal with the moments given by \eqref{momskr} and that given by 
\eqref{momstr} simultaneously,  we introduce,  under the condition \eqref{Dalang'}, for a positive $\rho\in [0, 1]$, 
\begin{align}\label{urho}
u^{\rho}(t,x):=
\E_X\Bigg[  u_0(X_{t}^{x})\exp \bigg(
\int_{0}^{t}\int_{\mathbb{R}^d} \delta_0
(X_{t-r}^{x}-y)W(dr,dy)\notag\\
\qquad\qquad -\frac\rho2\int_0^t\int_0^t |r-s|^{-\beta_0}\gamma(X_r-X_s)drds\bigg)\Bigg]\,. 
\end{align}
When $\rho=0,$ $u^\rho(t,x)$ is the Stratonovich solution  $u(t,x)$ to \eqref{pam}, and 
when $\rho=1$, $u^\rho(t,x)$ is the Skorohod  solution  $\tilde u(t,x)$ to \eqref{pam}.  
The $n$-th moment of $u^{\rho}(t,x)$ for a positive integer $n$ is given by
\begin{align}\label{mom}
\E[|u^\rho(t,x)|^n]=\E\Bigg[\prod_{j=1}^n &u_0(X_t^j+x)\exp\Bigg(\frac12\sum_{j,k=1}^n \int_0^t\int_0^t |r-s|^{-\beta_0}\gamma(X_r^j-X_s^k)drds\notag\\
&-\frac{\rho}2\sum_{j=1}^n\int_0^t\int_0^t |r-s|^{-\beta_0}\gamma(X_r^j-X_s^j)drds\Bigg)\Bigg].
\end{align} 
Let us point out that when $\rho=1$, $\E[|u^\rho(t,x)|^n]$ is finite under the weaker 
condition \eqref{Dalang}.  

The goal of this article is to obtain  the precise  asymptotics, as $t\rightarrow \infty$,  of the $p$-th moment $\E\left[ |u^\rho(t,x)|^p\right]$ for any (fixed) 
positive real number $p$.  To describe our main result,  we recall the definition of Fourier transform and introduce some notations. Denote by $\cS(\R^{d})$ the Schwartz space of smooth functions that are rapidly decreasing on $\R^{d}$, and let $\cS'(\R^{d})$ be its dual space, i.e., the space of tempered distributions.
  Let $\widehat f(\xi)$ or $(\F f)(\xi)$ denote  the Fourier transform of $f$,  for $f$ in the space  $\mathcal S'(\R^{d})$ of tempered distributions. In particular, we set
$$\widehat f(\xi) =\int_{\R^{d}} e^{-2\pi i x\cdot \xi}f(x)dx,\text{ for } f\in L^1(\R^{d}). $$
We will also need the following notations.
\begin{align}
&\mathcal E_\alpha(f,f):=\int_{\R^d}|\xi|^\alpha|\widehat f(\xi)|^2d\xi\,;\notag\\
&\mathcal F_{\alpha,d} :=\left\{f\in L^2(\R^d): \|f\|_2=1 \text{ and } ~\mathcal E_\alpha(f,f)<\infty\right\}\,; \label{Falphad}\\
&\mathcal A_{\alpha,d}:=\left\{g(s,x): \int_{\R^d} g^2(s,x)dx=1, \forall s\in[0,1] \text{ and } \int_0^1\int_{\R^d}|\xi|^\alpha |\widehat g(s,\xi)|^2d\xi ds<\infty\right\}\,;\label{Ad}\\
&\mathbf M(\alpha,\beta_0, d, \gamma):=\sup_{g\in \mathcal A_{\alpha,d}}\Bigg\{\frac12\int_0^1\int_0^1 \int_{\R^{2d}}\frac{\gamma(x-y)}{|r-s|^{\beta_0}}g^2(s,x)g^2(r,y)dxdydrds\notag\\
&\qquad \qquad \qquad\qquad \qquad -\int_0^1\int_{\R^d}|\xi|^\alpha|\widehat g(s,\xi)|^2d\xi ds\Bigg\}
\,. \label{M}
\end{align}
The finiteness of the variational representation $\mathbf M(\alpha,\beta_0, d, \gamma)$, when $\beta<\alpha$, is established in the Appendix. Note that $\mathbf M(\alpha,\beta_0, d, \gamma)$ has the scaling property, for any $\theta >0$,
\begin{equation}\label{mscaling}
\mathbf M(\alpha,\beta_0, d,\theta \gamma)=\theta^{\frac{\alpha}{\alpha-\beta}}\mathbf M(\alpha,\beta_0, d, \gamma),
\end{equation}
which can be derived in the same way as Lemma 4.1 in \cite{CHSX}. The following is the main result in this paper.
\begin{theorem}\label{pa}  Let $\rho\in[0,1]$ and   assume  the condition \eqref{Dalang'} (and when $\rho=1$, the condition \eqref{Dalang'} is replaced by  the condition \eqref{Dalang}).   Let $p\ge 2$ be any real number or $p=1$.  Then
\[\lim_{t\to\infty}t^{-\frac{2\alpha-\beta-\alpha\beta_0}{\alpha-\beta}}\log\|u^\rho(t,x)\|_p= (p-\rho)^{\frac{\alpha}{\alpha-\beta}}\mathbf M(\alpha, \beta_0, d,\gamma),\]
where $\|u^\rho(t,x)\|_p=\Big(\E[|u^{\rho}(t,x)|^p]\Big)^{1/p}$. 
\end{theorem}

 We conclude this  introduction with some remarks on the motivation of our work and a brief literature review for the related results. The following three 
points motivate  us to obtain the above asymptotics. 
\begin{enumerate}
\item[(i)] The limit related to the long-term asymptotics 
is known as the {\it moment Lyapunov exponent}
in literature and the problem is closely related to the issue of {\it intermittency} (see, e.g., \cite{K14}). 
To illustrate our idea, write the limit in Theorem \ref{pa} in the following form:
$$
\lim_{t\to\infty}t^{-\frac{2\alpha-\beta-\alpha\beta_0}{\alpha-\beta}}\log
\E\exp\Big(p\log |u^{\rho}(t,x)|\Big)=\Lambda (p).
$$
The system satisfies the usual definition of intermittency, i.e.,
the function $\Lambda(p)/p$ is strictly increasing on $[2,\infty)$.
By the large deviation theory (see, e.g., Theorem 1.1.4 in \cite{Chen}
and its proof for the lower bound), for any sufficiently large $l>0$
$$
\lim_{t\to\infty}t^{-\frac{2\alpha-\beta-\alpha\beta_0}{\alpha-\beta}}\log
\P(A_{t,l})=-\sup_{p>0}\Big\{lp-\Lambda(p)\Big\}<0
$$
and there is $p_l>0$ such that
$$
\lim_{t\to\infty}{\E\big[|u^{\rho}(t,x)|^{p_l}1_{A_{t,l}}\big]\over
\E[|u^{\rho}(t,x)|^{p_l}]}=1
$$
where 
$$
A_{t,l}=\Big\{\log |u^{\rho}(t,x)|\ge l
t^{\frac{2\alpha-\beta-\alpha\beta_0}{\alpha-\beta}}\Big\}.
$$
This observation shows that  as in other cases of 
 intermittency, it is  rare for the solution $u(t,x)$  to take large values but that
the impact of taking large values should not be ignored.

 \item[(ii)] When the noise $\dot W$ is the space-time white noise  with dimension one in space, the parabolic Anderson model  \eqref{pam} is  the model for the  {\it continuum directed polymer in random environment} (see \cite{AKQ14} for the case $\alpha=2$ and \cite{CSZ} for the case $\alpha<2$), where \eqref{pam}  is understood in the Skorohod sense, the solution $\tilde u(t,x)$ is the  {\it partition function} for the polymer measure, and $\log \tilde u(t,x)$ is the  {\it free energy} for the polymer (see, e.g., \cite{Corwin}). Similarly,  if we consider  an  $\alpha$-stable motion
  $X$ in the random environment modelled by $\dot W$, one may consider the 
 Hamiltonian  
 \begin{equation} \label{hrho}
 H^\rho(t,x):=\int_0^t \int_{\RR^d} \de_0(X_{t-r}^x-y) W(dr,dy)
 -\frac{\rho}{2} \int_0^t\int_0^t |r-s|^{-\beta_0}\gamma(X_r-X_s) drds\,. 
 \end{equation}
  Then, $u^{\rho}(t,x)=\E_X[e^{H^\rho(t,x)}]$ is the partition function for the polymer measure, and $\log u^{\rho}(t,x)$ is the free energy for the polymer.
\item[(iii)]  The equation \eqref{pam}, as one of the basic SPDEs, describes a variety of models, such as the parabolic Anderson model (see, e.g. \cite{CM94}) and the model for continuum directed polymer in random environment (see, e.g., \cite{AKQ14}), in which the long-term asymptotic property of the solution is desirable.  In the recent publication \cite{CHHH},   the space-time fractional diffusion equation of the form 
\[
\left(\partial^\beta  + \frac{\nu}{2} (-\Delta)^{\alpha/2} \right) u(t,x)= \lambda u(t,x) \dot{W}(t,x)\,,
\]      
has been studied, where $\partial^\beta$ is the Caputo derivative in time $t$.  It is highly non-trivial to
obtain precise asymptotics in general case. Our model \eqref{pam}  corresponds to 
the case $\beta=1$,  and our result may provide some perspective for the general situation.
\end{enumerate} 
The moment Lyapunov exponent has been studied extensively
with vast  literature. To our best knowledge, however,
the investigation  in the setting of white/fractional space-time Gaussian noise started only recently, especially at the level of precision
given in this paper.  When  the driving
processes are Brownian motion instead of stable process, i.e., the operator in \eqref{pam} is $\frac12 \Delta $ instead of the fractional Laplacian,  the long-term asymptotic lower and upper bounds for the moments of the solution were studied in \cite{BC16} for the Skorohod solution and in \cite{Song12} for the Stratonovich solution; 
the {\em precise} moment Lyapunov exponents were obtained in recent publications \cite{Chen'', Chen'}   for the Skorohod solutions,  and \cite{CHSX} for the Stratonovich solution.  In \cite{BC14}, the authors obtained  the intermittency property for the fractional heat equation in the Skorohod sense, by studying the lower and upper asymptotic bounds of the solution.

In the present paper,  we aim to obtain the {\em precise} $p$-th moment Lyapunov exponents for both Stratonovich solution and Skorohod solution to the fractional heat equation in a unified way, for any real positive number $p\ge2$. The mathematical challenges and/or the originality of this work come from the following aspects. First, compared with case of the heat equation, the fact that the fractional Laplacian is not a local operator  makes  the computations and analysis more sophisticated.  New ideas and methodologies are required. In particular, 
Fourier analysis is involved in a more substantial way.    Second,   the Feynman-Kac large deviation result for stable process (Theorem \ref{prop-fk}) is a key to our approach. However, the method used to derive a similar result for Brownian motion in  \cite{CHSX} can no longer  be applied, as the behavior of stable process is totally different from the behavior of Brownian motion.  Third, we obtain the precise long-term asymptotics for $u^{\rho}(t,x)$ with $\rho\in[0,1]$, which enables us to get the precise moment Lyapunov exponents for the Stratonovich solution and the Skorohod solution simultaneously. Finally, the  existing results on precise Lyapunov exponents were mainly for $n$-th moment with $n$  a positive integer,  due to the fact that the Feynman-Kac type representation is valid only for the moment of integer orders. We are able to extend the result from positive integers to real numbers $p\ge2$. The idea is to use   the variational inequality and the hypercontractivity of the Ornstein-Uhlenbeck semigroup operators.

The paper is organized as follows. In Section 2,  we establish some rough bounds for the long-term asymptotics of the Stratonovich solution by comparison method. The rough bounds will be used in the derivation of the precise upper bound in Section 6. The critical exponential integrability of $\int_0^1\int_0^1 |r-s|^{-\beta_0}\gamma(X_r-X_s)drds$ is also studied. In Section 3,  we obtain an Feynman-Kac type large deviation result for $\alpha$-stable processes, which plays a critical role in obtaining the variational representation for the precise moment Lyapunov exponent. In Section 4, 
 we establish a lower bound for the $p$-th moment of $u^{\rho}(t,x)$ 
 which is also valid if the $\alpha$-stable process is replaced by some general symmetric
  L\'evy process. In Sections 5 and 6, we  validate  the lower  bound and the upper bound 
  in Theorem \ref{pa}, respectively. Finally,
   in Appendix, the well-posedness of the variation given in \eqref{M} which appears in Theorem \ref{pa} is justified, and the proof of a technical lemma that is used in Section 6 is provided.

\section{Asymptotic bounds by comparison method}
In this section we derive long-term asymptotic bounds by comparison method for $\log\E[u(t,x)]\\=\log\E\exp\left(\int_0^t\int_0^t|r-s|^{-\beta_0}\gamma(X_r-X_s)drds\right)$.  Note that by the  self-similarity property of the stable process $X$, the integral inside the above exponential has the following scaling property,
\begin{align}\label{scaling}
 &\int_0^{at}\int_0^{at} |r-s|^{-\beta_0}\gamma(X_r-X_s)drds\nonumber\\
 \overset{d}{=}& a^{2-\frac\beta\alpha-\beta_0}\int_0^t\int_0^t|r-s|^{-\beta_0}\gamma(X_r-X_s)drds,\, \text{ for all } a>0.
\end{align}
First,  we present the following integrability result. 
\begin{lemma}\label{integrability} 
$\displaystyle \E\int_0^t\int_0^t|r-s|^{-\beta_0}\gamma(X_r-X_s)drds<\infty $ if and only if $\alpha\beta_0+\beta<\alpha.$
\end{lemma}
\begin{proof}
Using the self-similarity of $X$, and the scaling property of $\gamma(x)$, we have
$\E[\gamma(X_r-X_s)]=|r-s|^{-\frac\beta\alpha} \E[\gamma(X_1)],$ noting that $0<\E[\gamma(X_1)]<\infty,$ under the condition of this lemma. Hence,  we have
$$\E\int_0^t\int_0^t|r-s|^{-\beta_0}\gamma(X_r-X_s)drds =\E[\gamma(X_1)]\int_0^t\int_0^t|r-s|^{-\beta_0}|r-s|^{-\frac
\beta\alpha}drds,$$
which concludes the proof.\hfill\end{proof}

\begin{lemma}\label{continuity}  \rm{Under the condition (\ref{Dalang'})},  the process  
 \[
 Y_t=\int_0^t \int_0^ t |r-u|^{-\beta_0}\gamma(X_r-X_u) drdu,\quad  t\ge 0
\]
 has a continuous version.  
\end{lemma}
\begin{proof} We shall use the notation $\|F\|_p=\left(\E[|F|^p]\right)^{1/p}$.  
For any $0\le s<t\le \infty$,  we have for any $p\ge 1$, 
\begin{eqnarray*}
\|Y_t-Y_s\|_p 
&\le& \int_s^t \int_0^ s |r-u|^{-\beta_0}\|\gamma(X_r-X_u)\|_p drdu+
\int_0^s \int_s^ t |r-u|^{-\beta_0}\|\gamma(X_r-X_u)\|_p drdu\\
&&\qquad +\int_s^t \int_s^ t |r-u|^{-\beta_0}\|\gamma(X_r-X_u)\|_p drdu\\
&=:& I_1+I_2+I_3\,. 
\end{eqnarray*}
By scaling property,  when $1<p<\frac\alpha\beta$,  
\[
\|\gamma(X_r-X_u)\|_p=\left(\EE\left[ |X_r-X_u|^{-\frac\beta\alpha p}\right]\right)^{1/p}
=c_p|r-u|^{-\frac\beta\alpha}\,.
\]
Thus,
\begin{eqnarray*}
I_1
&\le&  \int_s^t \int_0^ s |r-u|^{-\beta_0-\frac\beta\alpha} dudr\\
&\le&  \frac{1}{  1-\beta_0-\frac\beta\alpha} 
 \int_s^t  r^{ 1-\beta_0-\frac\beta\alpha}  dr\\
 &\le&  C
 \int_s^t  t^{ 1-\beta_0-\frac\beta\alpha}  dr
 =C t ^ {1-\beta_0-\frac\beta\alpha} |t-s|\,. 
\end{eqnarray*} 
This means 
\[
I_1^p\le C  t ^ {p(1-\beta_0-\frac\beta\alpha)} |t-s|^p\,.
\]
 Similar estimates for $I_2$ and $I_3$ can also be obtained. Thus  for $0\le s,t\le T,$ there is a constant $C_T$ depending only on $(\alpha, \beta, \beta_0, T)$ such that
\[
\EE \left| Y_t-Y_s\right|^p=\|Y_t-Y_s\|_p ^p\le C_T |t-s|^p\,.
\]
It follows from Kolmogorov continuity criterion that $\left\{ Y_t\,, t\ge 0\right\}$ has a continuous version. 

\hfill
\end{proof}

\begin{theorem}\label{thm-exponential}
 Under the condition (\ref{Dalang'}), there exists a constant $\delta >0$ such that when $\theta \in(0,\delta)$, 
\begin{equation}\label{exponential'}
\E\exp\left(\theta\left(\int_0^1\!\!\int_0^1\vert
r-s\vert^{-\beta_0}\gamma( X_r-X_s)drds\right)^{\frac{\alpha}{\alpha\beta_0+\beta}}\right)<\infty,
\end{equation}
and consequently, for all $\lambda >0$,
\begin{align}\label{exponential}
\E\exp\bigg(\lambda\int_0^t\!\!\int_0^t\vert
r-s\vert^{-\beta_0}\gamma( X_r-X_s)drds\bigg)<\infty.
\end{align}
\end{theorem}
\begin{remark}
\rm{The inequality \eqref{exponential} is a consequence of Theorem 3.3 in \cite{Song} which was proved by using moment method.  Below we will provide another approach to prove \eqref{exponential} by using the techniques from the theory of large deviations,   and it turns out that this approach enables us to get a stronger result (see Remark \ref{remark-critical-exponent}).}
\end{remark}

\begin{proof} Denote 
\begin{equation}\label{eqzt}
Z_t=\left(\int_{0}^{t}\int_{0}^{t}|s-r|^{-\beta_0}\gamma(X_s-X_r)dsdr\right)^\frac{1}{2}\,,\text{ for }\,  t\ge 0\,.
\end{equation}
First we shall show that $Z_t$ is sub-additive and hence exponentially integrable by \cite[theorem 1.3.5]{Chen}.

The following  identity holds
\begin{equation}\label{timekernel}
\vert s-r\vert^{-\beta_0}=C_0\int_{\R}\vert s-u\vert^{-\frac{\beta_0+1}{2}}
\vert r-u\vert^{-\frac{\beta_0+1}{2}}du,
\end{equation}
where $C_0>0$ depends on $\beta_0$ only.
Similarly, for the function $\gamma(x)$ we also have
\begin{align}\label{spacekernel}
\gamma (x)=C(\gamma)\int_{\R^d}K(y-x)K(y)dy,\hskip.2in x\in\R^d\,, 
\end{align}
where $C(\gamma)>0$ is a constant and
\begin{eqnarray}\label{bound-9}
K(x)=\left\{\begin{array}{ll} \displaystyle\prod_{j=1}^d\vert
x_j\vert^{-{1+\beta_j\over 2}}&\qquad
\hbox{if  $\gamma(x)=\prod_{j=1}^d\vert x_j\vert^{-{\beta_j}}$},\\
\vert x\vert^{-{d+\beta\over 2}}&\qquad  
\hbox{if  $\gamma(x)= \vert x\vert^{-{\beta }}$}.
 \end{array}\right.
\end{eqnarray}
With these identities, we can rewrite $Z_t$ as 
$$
Z_t=\bigg(\int_{\R\times\R^d}\xi_t^2(u,x)
du dx\bigg)^{1/2}\,, 
$$
where
$$
\xi_t(u, x)=C_0C(\gamma)
\int_0^t\vert s-u\vert^{-{\beta_0+1\over 2}}K(X_s-x)ds\,. 
$$
For $t_1, t_2>0$, by the  triangular inequality
$$
Z_{t_1+t_2}\le Z_{t_1}+\bigg(\int_{\R\times\R^d}
\Big[\xi_{t_1+t_2}(u, x)-\xi_{t_1}(u, x)\Big]^2
du dx\bigg)^{1/2}\,. 
$$
Let $\widetilde{X}_s=X_{t_1+s}-X_{t_1}$, which is independent of $\{X_r,0\le r\le t_1\},$ and we have
\begin{align*}
&\xi_{t_1+t_2}(u, x)-\xi_{t_1}(u, x)\cr
=&C_0C(\gamma)\int_{t_1}^{t_1+t_2}
\vert s-u\vert^{-{\beta_0+1\over 2}}K(X_s-x)ds\cr
=&C_0C(\gamma)\int_{0}^{t_2}
\vert s+t_1-u\vert^{-{\beta_0+1\over 2}}K(\widetilde{X}_{s}+X_{t_1}-x)ds\,. 
\end{align*}
The translation invariance of the integral on $\R^{d+1}$ implies that
\begin{align*}
&\int_{\R\times\R^d}
\Big[\xi_{t_1+t_2}(u, x)-\xi_{t_1}(u, x)\Big]^2dudx 
= \int_{\R\times\R^d}\left[\widetilde{\xi}_{t_2}(u, x)\right]^2
du dx\,, 
\end{align*}
where
$$
\widetilde{\xi}_{t_2}(u, x)
=C_0C(\gamma)\int_{0}^{t_2}
\vert s-u\vert^{-{\beta_0+1\over2}}K(\widetilde X_s-x)ds.
$$
Therefore, the process $Z_t$ is sub-additive, which means that for any
$t_1,t_2>0$, $Z_{t_1+t_2}\le Z_{t_1}+\widetilde{Z}_{t_2}$, where
$\widetilde{Z}_{t_2}$ is independent of $\{Z_s, 0\le s\le t_1\}$
and has the same distribution as $Z_{t_2}$.

Notice that $Z_t$ is non-negative, non-decreasing, and  pathwise continuous by Lemma  \ref{continuity}. Thus it follows from \cite[Theorem 1.3.5]{Chen} that, for any  $t>0$ and  $\theta>0$
$$
\E\exp\left[\big( \theta Z_t\big)\right]<\infty,
$$
and 
\begin{equation}\label{subadditivity}
\lim_{t\to\infty}{1\over t}\log \E\left[\exp\big(\theta Z_t\big)\right]=\Psi(\theta),
\end{equation}
for some $\Psi(\theta)\in[0, \infty)$. Moreover, by the scaling property \eref{scaling} we have $Z_{at} \overset{d}{=}a^{\kappa} Z_t $ with $\kappa=1-\frac{\beta}{2\alpha}-\frac{\beta_0}{2}\in (1/2, 1)$, and hence for all $\theta>0$,
\begin{equation}\label{subadditivity'}\Psi(\theta)=
\lim_{t\to\infty}{1\over t}\log \E\exp\left[\big(\theta Z_t\big)\right]=\lim_{t\to\infty}{1\over t}\log \E\left[\exp\big( Z_{t\theta^{\frac1\kappa}} \big)\right]= \theta^{\frac1\kappa}\Psi(1).
\end{equation}
Chebyshev inequality implies that 
\[
 \exp(\theta t) \P(Z_t\ge t)\le \E\exp(\theta Z_t)
 \quad\hbox{and then}\quad 
 \theta t+\log \P(Z_t\ge t)\le \log \E\exp(\theta Z_t)\,.
 \]
 Taking the  limit yields, for any $\theta>0$,
\[\limsup_{t\to \infty} \frac1t \log\P(Z_t\ge t)\le \lim_{t\to \infty}\frac1t \log \E\left[\exp(\theta Z_t)\right]-\theta=\theta^{\frac1\kappa}\Psi(1)-\theta.\]
Therefore
\begin{equation}\label{eq2.9'}
\limsup_{t\to \infty} \frac1t \log\P(Z_t\ge t)\le \inf_{\theta\in(0,1)}\{\theta^{\frac1\kappa}\Psi(1)-\theta\},\end{equation}
where the term on the right-hand side is strictly negative noting that $1/\kappa\in(1,2)$ and $\Psi(1)\ge 0$, and is denoted by $-a$ for some $a>0$. Hence there exists a constant $ T>0$ such that when $t\ge T$, 
\begin{equation}\label{eq2.10'}
\P(Z_1\ge t^{1-\kappa})=\P(Z_t\ge t)\le \exp\left(-  at/2 \right).
\end{equation}
Consequently, 
\begin{eqnarray*}
 \E\left[\exp(\theta Z_1^{\frac1{1-\kappa}})\right]
 &=&\int_0^\infty \P(\theta Z_1^{\frac1{1-\kappa}}\ge y) e^y  dy +1\\
&\le&  \int_0^T e^y dy +\int_T^\infty e^{- a\theta^{-1}y/2} e^y  dy+1,
\end{eqnarray*}
where the last term is finite if $\theta \in (0, a/2)$. This implies \eref{exponential'}. 

Finally \eref{exponential} is obtained by  \eref{exponential'}, the scaling property \eref{scaling} and the fact that the condition \eqref{Dalang'} implies  $  \frac{\alpha}{\alpha \beta_0+\beta}>1$.
\end{proof}

\begin{remark}\label{remark-exponential}  \rm{Note that by\eref{subadditivity'}, $\Psi(\theta)=\theta^{\frac1\kappa}\Psi(1)$ with $\Psi(1)\in[0, \infty)$. Actually, $\Psi(1)> 0$ when $\beta_0=0$, by \eqref{bound-6} in the proof of Lemma \ref{bound-4}. However, when $\beta_0\in(0,1)$, $\Psi(1)$ must be 0,  which means that the asymptotics given by \eqref{subadditivity} is not optimal. Indeed, if $\Psi(1)\neq 0$, G\"artner-Ellis theorem for non-negative random variable (\cite[Corollary 1.2.5]{Chen})  and equation \eref{subadditivity'} imply that for $\lambda>0$,
\[
\lim_{t\to\infty}\frac1t \log\P(Z_t^2\ge\lambda t^{2-2\kappa})=-\sup_{\theta>0}\big\{\theta \lambda^\frac12-\theta^{\frac1\kappa}\Psi(1)\big\}= C_1 \lambda^{\frac{1}{2-2\kappa}},\
\]
where $C_1= \Psi(1)^{\frac{\kappa}{\kappa-1}}(\kappa^{\frac{\kappa}{1-\kappa}}-\kappa^{\frac{1}{1-\kappa}})$. Note that the assumption $\Psi(1)>0$ guarantees that $\theta^{\frac1\kappa}\Psi(1)$ is an essentially smooth function (\cite[Definition 1.2.3]{Chen}), and hence the G\"artner-Ellis theorem can be applied. Then, by the Varadhan's integral lemma (\cite[Theorem 1.1.6]{Chen} or \cite[Section 4.3]{DZ}), 
\[
\lim_{t\to\infty} \frac1t\log \E\exp(\theta t^{2\kappa -1}Z_t^2)=\sup_{\lambda>0}\big\{\lambda \theta -C_1 \lambda^{\frac1{2-2\kappa}}\big\}=C_2 \theta^{\frac1{2\kappa-1}}\,,
\]
where $C_2$ is a positive constant depending on $C_1$ and $\kappa$. By the scaling property \eref{scaling}, this limit is equal to
\[
\lim_{t\to \infty}t^{-1}\log\E\left[\exp(\theta Z_{t^{\eta}}^2)\right]=\lim_{t\to \infty} 
t^{-\frac1\eta}\log\E\left[\exp(\theta Z_{t}^2)\right]=C_2 \theta^{\frac{1}{2\kappa-1}} \,,
\]
where $\eta=\frac{2\kappa-1}{2\kappa}$ and $\frac1\eta=\frac{2\alpha-\beta-\alpha\beta_0}{\alpha-\beta-\alpha\beta_0}$. This contradicts with Proposition \ref{bound-1} when $\beta_0\in(0,1)$.  }
\end{remark}

\begin{remark}\label{remark-critical-exponent}\rm{
 We observe that the restriction $\theta\in (0,\delta)$ for \eqref{exponential'} in Theorem \ref{thm-exponential} can be removed when $\beta_0\in(0,1)$. Indeed, the inequality \eqref{eq2.9'} in the proof can be replaced by
\begin{equation*}
\limsup_{t\to \infty} \frac1t \log\P(Z_t\ge t)\le \inf_{\theta>0}\{\theta^{\frac1\kappa}\Psi(1)-\theta\}.\end{equation*}
Noting that by Remark \ref{remark-exponential}, $\Psi(1)=0$ when $\beta_0\in(0,1)$, we have 
$$\limsup_{t\to \infty} \frac1t \log\P(Z_t\ge t)= -\infty. $$
This enables us to choose any positive number for $a$ in \eqref{eq2.10'}, and hence \eqref{exponential'} holds for any $\theta>0$. Moreover, using Theorem \ref{pa} (note that Theorem \ref{pa} is proved without quoting Theorem \ref{thm-exponential'}), the critical exponential integrability and the corresponding critical exponent for $\int_0^1\!\!\int_0^1\vert
r-s\vert^{-\beta_0}\gamma( X_r-X_s)drds $ can be obtained. 
\begin{theorem}\label{thm-exponential'}
Let $C_0:=C\big(\alpha, \beta,\beta_0, d,\gamma(\cdot)\big)$ be given in \eqref{eq.critical}.
Then under the condition \eqref{Dalang'}, we have
\begin{equation}\label{eq.cr.1}
\E\exp\left(\theta\left(\int_0^1\!\!\int_0^1\vert
r-s\vert^{-\beta_0}\gamma( X_r-X_s)drds\right)^{\frac{\alpha}{\beta}}\right)<\infty, \text{ for any $\theta<C_0$, }
\end{equation}
and 
\begin{equation}\label{eq.cr.2}
\E\exp\left(\theta\left(\int_0^1\!\!\int_0^1\vert
r-s\vert^{-\beta_0}\gamma( X_r-X_s)drds\right)^{\frac{\alpha}{\beta}}\right)=\infty, \text{ for any $\theta>C_0$.}
 \end{equation}
 Furthermore, 
 \begin{equation}\label{eq.cr.3}
\E\exp\left(\theta\left(\int_0^1\!\!\int_0^1\vert
r-s\vert^{-\beta_0}\gamma( X_r-X_s)drds\right)^{\eta}\right)<\infty, \text{ for any $\theta>0$ and $\eta<\frac\alpha\beta$.}
\end{equation}
and 
 \begin{equation}\label{eq.cr.4}
\E\exp\left(\theta\left(\int_0^1\!\!\int_0^1\vert
r-s\vert^{-\beta_0}\gamma( X_r-X_s)drds\right)^{\eta}\right)=\infty, \text{ for any $\theta>0$ and $\eta>\frac\alpha\beta$.}
\end{equation}

\end{theorem}
\begin{proof}
Recall that $Z_t$ is defined in \eqref{eqzt}. Theorem \ref{pa} implies that, when $p=1$ and $\rho=0$, 
$$\lim_{t\to\infty}t^{-\frac{2\alpha-\beta-\alpha\beta_0}{\alpha-\beta}} \log \E \exp\left(\frac12 Z_t^2\right)=\mathbb M(\alpha, \beta_0, d,\gamma).$$
By the scaling property \eqref{scaling} of $Z_t^2$ and the change of variable $s=t^{\frac{2\alpha-\beta-\alpha\beta_0}{\alpha-\beta}}$, the above equation is equivalent to 
$$\lim_{s\to \infty }\frac1s \log \E\exp\left( \theta  s^{1-\beta/\alpha} Z_1^2 \right)=(2\theta)^{\frac{\alpha}{\alpha-\beta}}\mathbb M(\alpha, \beta_0, d,\gamma).$$
Then the G\"artner-Ellis theorem implies 
\begin{align*}
\lim_{s\to \infty}\frac1s \log \mathbb P(s^{-\beta/\alpha} Z_1^2\ge \lambda)&=-\sup_{\theta>0}\big\{\theta\lambda -(2\theta)^{\frac{\alpha}{\alpha-\beta}}\mathbb M(\alpha, \beta_0, d,\gamma) \big\}\\
&=-\lambda^{\frac\alpha\beta} \frac{\beta}{\alpha-\beta} \left(\frac{\alpha-\beta}{2\alpha}\right)^{\frac\alpha\beta}  \Big( \mathbb M(\alpha, \beta_0, d,\gamma)\Big)^{\frac{\beta-\alpha}{\beta}}.
\end{align*}
Denote
\begin{equation}\label{eq.critical}
C_0:= \frac{\beta}{\alpha-\beta} \left(\frac{\alpha-\beta}{2\alpha}\right)^{\frac\alpha\beta}  \Big( \mathbb M(\alpha, \beta_0, d,\gamma)\Big)^{\frac{\beta-\alpha}{\beta}},
\end{equation}
and hence $C_0$ is a finite positive constant. 
Then we have 
\begin{equation}\label{eq.zlimit}
\lim_{s\to \infty}\frac1s \log \mathbb P( Z_1^\frac{2\alpha}\beta\ge s)=-C_0.
\end{equation}
For any fixed $\si\in(0,1)$, there exists $T_\si>0$ such that when $t>T_\si$, 
$$\mathbb P(Z_1^\frac{2\alpha}\beta\ge t)\le \exp(-t\si C_0), $$ and therefore,
\begin{eqnarray*}
 \E\left[\exp(\theta Z_1^{\frac{2\alpha}\beta})\right]
 &=&1+\int_0^\infty \P(\theta Z_1^{\frac{2\alpha}\beta}\ge t) e^t  dt\\
&\le& 1+ \int_0^{T_\si} e^t dt +\int_{T_\si}^\infty e^{- t\si \theta^{-1}C_0} e^t  dy,
\end{eqnarray*}
where the right-hand side is finite when $\theta<\si C_0$.  Since  $\si\in (0,1)$ can be arbitrarily chosen, the first result \eqref{eq.cr.1} is obtained.  Finally the inequalitys \eqref{eq.cr.2}, \eqref{eq.cr.3}, \eqref{eq.cr.4} can be proved in a similar way by using \eqref{eq.zlimit}, and the proof is concluded.
\end{proof}
}
\end{remark}

\medskip

To obtain the optimal asymptotics for 
 $\E\exp\left(\int_0^t\int_0^t|r-s|^{-\beta_0}\gamma(X_r-X_s)drds\right)$,   we first investigate the asymptotics of $ \E\exp\left(\int_0^t\int_0^t \gamma(X_s-X_r)dsdr\right)$.
\begin{lemma}\label{bound-4} Under the  condition \eqref{Dalang}, there exists a 
constant $C\in(0,\infty)$, such that 
\begin{equation}\label{gamma}
\lim_{t\to\infty} t^{-\frac{2\alpha-\beta}{\alpha-\beta}} \log \E \exp\left( \theta\int_0^t\int_0^t \gamma(X_r-X_s)drds\right)=C \theta^{\frac{\alpha}{\alpha-\beta}}, \quad \forall\, \theta>0.
\end{equation}
 Let $\tilde X$ be an independent copy of $X$. Then under the condition \eqref{Dalang}, there exist $0<D_1\le D_2<\infty$ such that for all $\theta>0$,
\begin{align}\label{gamma'}
D_1\theta^{\frac{\alpha}{\alpha-\beta}}&\le \liminf_{t\to\infty} t^{-\frac{2\alpha-\beta}{\alpha-\beta}} \log \E \exp\left( \theta\int_0^t\int_0^t \gamma(X_r-\tilde X_s)drds\right)\notag\\
&\le \limsup_{t\to\infty} t^{-\frac{2\alpha-\beta}{\alpha-\beta}} \log \E \exp\left( \theta\int_0^t\int_0^t \gamma(X_r-\tilde X_s)drds\right)\le D_2\theta^{\frac{\alpha}{\alpha-\beta}}.
\end{align}

\end{lemma}

\begin{proof}  When $\gamma(x)=|x|^{-\beta}$, \eqref{gamma} is a direct consequence of \cite[Equation (1.18) ]{CR} using the scaling property of the $\int_0^t\int_0^t \gamma(X_r-X_s)drds$. 
When $\gamma(x)=\prod_{j=1}^d|x_j|^{-\beta_j}$, it suffices to show that there there exists a constant $C_1<\infty$ such that
\begin{equation}\label{gamma1}
\lim_{t\to\infty} t^{-\frac{2\alpha-\beta}{\alpha-\beta}} \log \E \exp\left( \theta\int_0^t\int_0^t \gamma (X_r-X_s)drds\right)=C_1 \theta^{\frac{\alpha}{\alpha-\beta}}\,. 
\end{equation}
This is because   that $\prod_{j=1}^d|x_j|^{-\beta_j}\ge|x|^{-\beta}$ and hence 
$C_1$ is greater  than or equal to the constant $C>0$ in \eqref{gamma} 
when  $\gamma(x)=|x|^{-\beta}$.  This means that if $C_1<\infty$ satisfies \eqref{gamma1}, 
then it will be automatically positive. 

From now on  the generic constant $C$ may be different in different places.

We claim that (\ref{gamma1}) is equivalent to 
\begin{align}\label{bound-6}
\lim_{t\to\infty}{1\over t}\log\E
\exp\bigg(\theta\bigg(\int_0^t\!\!\int_0^t\gamma
(X_r-X_s)drds\bigg)^{1/2} \bigg) =C\theta^{2\alpha\over 2\alpha-\beta}\,,
\hskip.2in \ \forall \  \theta>0\,
\end{align}
for some constant $C\in(0,\infty)$, which can be proved in the same way as we did to get \eqref{subadditivity'} in the proof of the Theorem \ref{thm-exponential}. Indeed, by the scaling property (\ref{scaling}) with $\beta_0=0$,
and by a G\"artner-Ellis type result for non-negative random
variables (\cite[Corollary 1.2.5]{Chen}), both (\ref{gamma1})
and (\ref{bound-6}) are equivalent to the tail asymptotics
$$
\lim_{t\to\infty}t^{-1}\log\P\bigg(\int_0^1\!\!\int_0^1\gamma(
X_r-X_s) drds\ge\lambda
t^{\frac\beta\alpha}\bigg)=-\sup_{\theta>0}
\left\{\sqrt\lambda \theta-C\theta^{\frac{2\alpha}{2\alpha-\beta}}\right\}=-C\lambda^{\frac\alpha\beta}\,, \,\,\forall \  \lambda>0.
$$

 Now we prove \eqref{gamma'}. The upper bound can be obtained by \eqref{gamma} and the observation that 
\begin{align*}
&\E \exp\left( \theta\int_0^t\int_0^t \gamma(X_r-\tilde X_s)drds\right)=\E \exp\left( \theta C(\gamma)\int_{\R^d}\int_0^t K(X_r-x)dr\int_0^t K(\tilde X_s-x)dsdx\right)\\
&\le \E \exp\left( \frac \theta2 C(\gamma)\int_{\R^d}\left(\int_0^t K(X_r-x)dr\right)^2 +\left(\int_0^t K(\tilde X_r-x)dr\right)^2dx\right)\\
&= \left[ \E \exp\left( \frac \theta2 C(\gamma)\int_{\R^d}\left(\int_0^t K(X_r-x)dr\right)^2 dx\right)\right]^2\le \E \exp\left( \theta\int_0^t\int_0^t \gamma(X_r-X_s)drds\right).
\end{align*}
 For the lower bound, if suffices to consider the case $\gamma(x)=|x|^{-\beta}$. By \cite[Theorem 1.2]{BCR}, 
 $$\lim_{t\to\infty} t^{-\frac\alpha\beta}\log \mathbb P\left(\int_0^1\int_0^1\gamma(X_r-\tilde X_s)drds\ge t\right)=-a,$$
 where $a$ is a positive constant. By the scaling property \eqref{scaling}, the above equality is equivalent to 
 $$\lim_{t\to\infty} t^{-\frac{2\alpha-\beta}{\alpha-\beta}}\log \mathbb P\left(t^{-\frac{2\alpha-\beta}{\alpha-\beta}}\int_0^t\int_0^t\gamma(X_r-\tilde X_s)drds\ge \lambda\right)=-a\lambda^{\frac\alpha\beta}, \text{ for all } \lambda>0.$$
 Then by Varadhan's integral lemma, we have 
 $$\lim_{t\to \infty}  t^{-\frac{2\alpha-\beta}{\alpha-\beta}}\log \exp\left(\theta\int_0^t\int_0^t\gamma(X_r-\tilde X_s)drds\right)=\sup_{\lambda>0}\{\theta \lambda-a\lambda^{\frac\alpha\beta}\} = b \theta^{\frac{\alpha}{\alpha-\beta}}, $$
 for some $b>0.$
\end{proof}
\medskip

Based on the above result, we shall derive the following asymptotics for $\E\exp\Big(\theta\int_0^t\!\!\int_0^t\vert
r-s\vert^{-\beta_0}\gamma( X_r-X_s)drds\Big)$ by comparison method.
\begin{proposition}\label{bound-1}
Under the condition \ref{Dalang'},  there is $0<C_1<C_2<\infty$ such that for any $\theta>0$,
\begin{align}\label{bound-3}
C_1\theta^{{
\alpha\over \alpha-\beta}}&\le \liminf_{t\to\infty}t^{-{2\alpha-\beta-\alpha\beta_0\over \alpha-\beta}}\log
\E\exp\bigg(\theta\int_0^t\!\!\int_0^t\vert
r-s\vert^{-\beta_0}\gamma( X_r-X_s)drds\bigg)\notag\\
&\le \limsup_{t\to\infty}t^{-{2\alpha-\beta -\alpha\beta_0\over \alpha-\beta}}\log
\E\exp\bigg(\theta\int_0^t\!\!\int_0^t\vert
r-s\vert^{-\beta_0}\gamma( 
X_r-X_s)drds\bigg)\le C_2\theta^{{ \alpha\over
\alpha-\beta}}.
\end{align}
 Similarly, under the condition \eqref{Dalang}, there is $0<D_1<D_2<\infty$ such that for any $\theta>0$,
\begin{align}\label{bound-3''}
D_1\theta^{{
\alpha\over \alpha-\beta}}&\le \liminf_{t\to\infty}t^{-{2\alpha-\beta-\alpha\beta_0\over \alpha-\beta}}\log
\E\exp\bigg(\theta\int_0^t\!\!\int_0^t\vert
r-s\vert^{-\beta_0}\gamma( X_r-\tilde X_s)drds\bigg)\notag\\
&\le \limsup_{t\to\infty}t^{-{2\alpha-\beta -\alpha\beta_0\over \alpha-\beta}}\log
\E\exp\bigg(\theta\int_0^t\!\!\int_0^t\vert
r-s\vert^{-\beta_0}\gamma( 
X_r-\tilde X_s)drds\bigg)\le D_2\theta^{{ \alpha\over
\alpha-\beta}}.
\end{align}
\end{proposition}

\begin{remark}
By the scaling property \eqref{scaling}, the above asymptotics \eqref{bound-3} is equivalent to 
\begin{align}\label{bound-3'}
C_1\theta^{{\alpha\over \alpha-\beta}}&\le \liminf_{t\to\infty} \frac1t \log \E\exp\bigg(\frac\theta t\int_0^t\!\!\int_0^t\vert
r-s\vert^{-\beta_0}\gamma( X_r-X_s)drds\bigg)\notag\\
&\le 
\limsup_{t\to\infty} \frac1t \log \E\exp\bigg(\frac\theta t\int_0^t\!\!\int_0^t\vert
r-s\vert^{-\beta_0}\gamma( X_r-X_s)drds\bigg)\le C_2\theta^{{\alpha\over \alpha-\beta}},
\end{align}
respectively.  We also have a similar result for \eqref{bound-3''}. 
\end{remark}
\begin{proof}  The proof is similar to \cite[Proposition 2.1]{CHSX}, but we include details for the reader's convenience. First we prove the lower bound in \eqref{bound-3}. 
Note that 
$$
\begin{aligned}
&\int_0^t\!\!\int_0^t\vert r-s\vert^{-\beta_0}\gamma(
X_r-X_s)drds\ge t^{-\beta_0}
\int_0^t\!\!\int_0^t\gamma(
X_r-X_s)drds,
\end{aligned}
$$
where the term on the right-hand side has the same distribution as $$
\int_0^{t^{2\alpha-\beta-\alpha\beta_0\over 2\alpha-\beta}}\!
\int_0^{t^{2\alpha-\beta-\alpha\beta_0\over 2\alpha-\beta}}\gamma(
X_r-X_s)drds$$
by the scaling property \eref{scaling}. 
Then the lower bound  is an immediate consequence of
Lemma \ref{bound-4}.

Now we show the upper bound of \eqref{bound-3}.
By the symmetry of the integrand function,  we have 
$$
\int_0^t\!\!\int_0^t\vert r-s\vert^{-\beta_0}\gamma(
X_r-X_s)drds =2\iint_{[0\le s\le r \le t]}\vert r-s\vert^{-\beta_0}\gamma(
X_r-X_s)drds\,. 
$$
Thus, the  inequality  (\ref{bound-3}) is  equivalent to
\begin{align}\label{bound-11}
\limsup_{t\to\infty}t^{-{2\alpha-\beta-\alpha\beta_0\over \alpha-\beta}}\log
\E\exp\bigg(\theta\iint_{[0\le s\le r\le t]}
\vert r-s\vert^{-\beta_0}\gamma(
X_r-X_s)drds\bigg)\le C\theta^{{  \alpha\over \alpha-\beta}}\,.
\end{align}

Compared with lower bound, the estimation \eref{bound-11} is more difficult to obtain because $|r-s|^{-\beta_0}$ is unbounded when $r$ and $s$ are close. We shall decompose the integral $\int_{[0\le s\le r\le t]}
\vert r-s\vert^{-\beta_0}\gamma(
X_r-X_s)drds$ and then apply H\"older inequality to obtain the desired result. More precisely, let $[0\le s\le r\le t]=I_1\cup I_2 \cup I_3$, where $I_1=[0\le s\le t\le t/2], I_2=[t/2\le s\le r\le t]$ and $I_3=[0,t/2]\times [t/2, t]$. Noting that $\iint_{I_1}|r-s|^{-\beta_0}\gamma(X_r-X_s)drds$ and $\iint_{I_2}|r-s|^{-\beta_0}\gamma(X_r-X_s)drds$ are mutually independent and 
 are equal in law, by the H\"older inequality,
$$
\begin{aligned}
&\E\exp\bigg(\theta\iint_{[0\le s\le r \le t]}
\vert r-s\vert^{-\beta_0}\gamma(
X_r-X_s)drds\bigg)\\
&\le\Bigg(\E\exp\bigg(\theta p\iint_{I_1}
\vert r-s\vert^{-\beta_0}\gamma(
X_r-X_s)drds\bigg)\Bigg)^{2/p}\\
&\quad\quad \times\Bigg(\E\exp\bigg(\theta q
 \iint_{I_3} \vert
r-s\vert^{-\beta_0}\gamma( X_r-X_s)drds\bigg)\Bigg)^{1/q}\,,
\end{aligned}
$$
where $p^{-1}+q^{-1}=1$. Furthermore, by the scaling property \eqref{scaling},
$$
\begin{aligned}
&\iint_{I_1}
\vert r-s\vert^{-\beta_0}\gamma(
X_r-X_s)drds\\
&\buildrel d\over =\Big({1\over 2}\Big)^{2\alpha-\beta-\alpha\beta_0\over \alpha}
\iint_{[0\le s\le r\le t]}
\vert r-s\vert^{-\beta_0}\gamma(
X_r-X_s)drds\,.
\end{aligned}
$$
Taking $p=2^{2\alpha-\beta-\alpha\beta_0\over \alpha}$, we have
$$
\begin{aligned}
&\E\exp\bigg(\theta\iint_{[0\le r\le s\le t]}
\vert r-s\vert^{-\beta_0}\gamma(
X_r-X_s)drds\bigg)\\
&\le \Bigg(\E\exp\bigg(\theta q
\int_0^{t/2}\!\!\int_{t/2}^t
|r-s|^{-\beta_0}\gamma(
X_r-X_s)drds\bigg)\Bigg)^{{1\over q}{p\over p-2}}\,.
\end{aligned}
$$
Now to obtain \eqref{bound-3}, it suffices to show  
\begin{align}\label{bound-12}
\limsup_{t\to\infty}t^{-{2\alpha-\beta -\alpha\beta_0\over \alpha-\beta}}\log
\E\exp\bigg(\theta
\int_0^{t/2}\!\!\int_{t/2}^{t}
|r-s|^{-\beta_0}\gamma(
X_r-{X}_s)drds\bigg)\le C\theta^{{  \alpha\over \alpha-\beta}}\,.
\end{align}
 Actually, decomposing $[0, t/2]\times[t/2 , t]$ as $A\cup B$, where $A=[t/4, t/2]\times [t/2, 3t/4]$ and $B=[0, t/2]\times[t/2 , t]\setminus A$, we have 
 \begin{align}\label{e.2.22}
&\E\exp\bigg(\theta\int_{0}^{t/2}\int_{t/2}^t
\vert r-s\vert^{-\beta_0}\gamma(
X_r-X_s)drds\bigg)\notag\\
&\le \Bigg(\E\exp\bigg(\theta p \iint_{A}
\vert r-s\vert^{-\beta_0}\gamma(
X_r-X_s)drds\bigg)\Bigg)^{1/p}\notag\\
&\quad \quad \times\Bigg(\E\exp\bigg(\theta q
 \iint_{B} \vert
r-s\vert^{-\beta_0}\gamma( X_r-X_s)drds\bigg)\Bigg)^{1/q}\,,
\end{align}
where $1/p+1/q=1$ and $p,q>0$ are to be determined later. Since $X$ has stationary increments and by (\ref{scaling}),  we have
$$
\begin{aligned}
\iint_{A}
\vert r-s\vert^{-\beta_0}\gamma(
X_r-X_s)drds &\buildrel d\over =
\int_0^{t/4}\!\!\int_{t/4}^{t/2} 
\vert r-s\vert^{-\beta_0}\gamma(
X_r-X_s)drds\\
&\buildrel d\over =\Big({1\over 2}\Big)^{2\alpha-\beta-\alpha\beta_0\over \alpha}
\int_0^{t/2}\!\!\int_{t/2}^t
\vert r-s\vert^{-\beta_0}\gamma(
X_r-X_s)drds\,.
\end{aligned}
$$
Now let us choose $p=2^{2\alpha-\beta-\alpha\beta_0\over \alpha}$, and the above identity combined with
 \eref{e.2.22} yields 
\begin{align*}
&\E\exp\bigg(\theta\int_{0}^{t/2}\!\!\int_{t/2}^t
\vert r-s\vert^{-\beta_0}\gamma(
X_r-X_s)drds\bigg)\le \E\exp\bigg(\theta q
 \iint_{B} \vert
r-s\vert^{-\beta_0}\gamma( X_r-X_s)drds\bigg)\\
&\le \E\exp\bigg(\theta q \left(\frac{t}{4}\right)^{-\beta_0}
 \int_0^t\!\!\int_0^t \gamma( X_r-X_s)drds\bigg)
\le \E\exp\bigg(\theta q 4^{\beta_0}
 \int_0^{t^{\eta}}\!\!\int_0^{t^{\eta}} \gamma( X_r-X_s)drds\bigg)\,,
\end{align*}
where $\eta={2\alpha-\beta-\alpha\beta_0\over
2\alpha-\beta}$. Thus \eref{bound-12} follows from Lemma \ref{bound-4} with $t$
being replaced by $t^{\eta}$ and \eqref{bound-3}  is obtained.  

 The lower bound in \eqref{bound-3''} can be obtained in a similar way as for the lower bound in \eqref{bound-3}, by using the second half of Lemma \ref{bound-4}.  Now we show the upper bound. Noting that the stable process has  stationary increments which are independent over disjoint time intervals, we have
\begin{align*}
\int_{0}^{t/2}\!\!\int_{t/2}^t
\vert r-s\vert^{-\beta_0}\gamma(
X_r-X_s)drds \overset{d}{=}\int_{0}^{t/2}\!\!\int_0^{t/2}
\vert r-s\vert^{-\beta_0}\gamma(
X_r-\tilde X_s)drds. 
\end{align*}
By  Remark 5.7 in \cite{Song}, under the condition \eqref{Dalang},
$$ \E\exp\left(\lambda\int_{0}^{t}\!\!\int_0^{t}
\vert r-s\vert^{-\beta_0}\gamma(
X_r-\tilde X_s)drds\right)<\infty \text{ for all } \lambda >0. $$
Hence \eqref{bound-12} still holds under the condition \eqref{Dalang}, and therefore the upper bound in \eqref{bound-3''} is obtained. The proof is concluded. 
\end{proof}

\section{Feynman-Kac large deviation for stable process}

 In this section, we will obtain a Feynman-Kac large deviation result (Proposition \ref{prop-fk} below) for symmetric $\alpha$-stable process, which is a space-time extension of Lemma 6 in \cite{CLR} and will play a critical role in the derivation of our main result.   In \cite{CHSX} a similar result for Brownian motion was obtained (Proposition 3.1 in that paper) in order to get the precise moment Lyapunov exponent for the Stratonovich solution of heat equation.   The approach in \cite{CHSX} heavily depends on the local property of the Laplacian operator and the property of Brownian motion such as the continuity of paths and the Gaussian tail probability, and hence cannot be adapted to our situation, as the fractional Laplacian is a non-local operator, the stable process is a pure jump process, and the stable distribution is fat-tailed. Inspired by the idea in \cite{CLR}, instead of considering the stable process itself, we shall consider the stable process restricted in bounded domains by taking its image of quotient map, which will be elaborated below.

 Fix a positive number $M$. Let $\mathbb T_M^d=\R^d/M\mathbb Z^d$ be the $d$-dimensional torus  and $X^M_t$ be the image of $X_t$ under the quotient map from $\R^d$ to  $\mathbb T_M^d$. Then, $X^M$ is a Markov process with independent increments on $\mathbb T_M^d$, and its associated Dirichlet form is given by 
\begin{equation}\label{dirichlet}
\mathcal E_{\alpha,M}(f,f):= \frac1{M^{d+\alpha}}\sum_{k\in \mathbb Z^d}|k|^\alpha|\widehat f(k)|^2,
\end{equation}
where $\widehat f$ denotes the usual Fourier transform for functions on $\mathbb T_M^d$, i.e., for $k\in \mathbb Z^d$,
$$ \widehat f(k):=\int_{\mathbb T_M^d} f(x)e^{-2\pi i k\cdot x/M} dx=\int_{[0,M]^d} f(x)e^{-2\pi i k\cdot x/M} dx.$$
Here the function $f$ on $\mathbb T_M^d$ is considered as an $M$-periodic function (with the same symbol $f$) on $\R^d$, which means that $f(x+kM)=f(x)$ for any $k\in \mathbb Z^d$.
Let 
\begin{equation}
\mathcal F_{\alpha,M}:=\{f\in L^2(\mathbb T_M^d): \|f\|_{2,\mathbb T_M^d}=1 \text{ and } \mathcal E_{\alpha, M}(f,f)<\infty\},
\end{equation}
where $$\|f\|_{2,\mathbb T_M^d}=\left(\langle f,f\rangle_{2,\mathbb T_M^d}\right)^{1/2}:=\left(\int_{\mathbb T_M^d} |f(x)|^2dx\right)^{1/2}=\left(\int_{[0,M]^d}|f(x)|^2dx\right)^{1/2}$$ is the $L^2$-norm on $\mathbb T_M^d$ endowed with the Lebesgue measure.

\begin{proposition}\label{prop-fk}
Let $f(t,x): [0,1]\times\mathbb T_M^d\to \R$ be a continuous function.  Then, we have 
\begin{equation}\label{fk}
\lim_{t\to\infty} \frac1t\log\E\left[\exp\left(\int_0^t f(\frac st, X_s^M)ds \right)\right]=\int_0^1 \lambda_M(f(s,\cdot))ds,
\end{equation}
where $\lambda_M(f):=\sup\limits_{g\in\mathcal F_{\alpha,M}}\Big\{\langle g,fg\rangle_{2,\mathbb T_M^d}-\mathcal E_{\alpha, M}(g, g)\Big\}.$
\end{proposition}
\begin{proof}
Let $\{0=s_0<s_1<\dots<s_{n-1}<s_{n}=1\}$ be a uniform partition of the interval $[0,1]$. First, we consider the functions of the form 
\[
f(s,x)=\sum_{i=0}^{n-1} f_i(x) I_{[s_i, s_{i+1})}(s)+f_{n-1}(x)I_{\{1\}}(s)\,.
\]
By the Markov property, we have 
\begin{align*}
&\E\left[\exp\left(\int_0^{t} f(\frac st,X_s^M)ds\right)\right]=\E\left[\exp\left(\int_0^{\frac tn} f(\frac st,X_s^M)ds\right)\exp\left(\int_{\frac tn}^{t} f(\frac st,X_s^M)ds\right)\right]\\
&=\E\left[\exp\left(\int_0^{\frac tn} f_0(X_s^M)ds\right)\E_{X_{\frac tn}^M}\left[\exp\left(\int_0^{(1-\frac1n)t} f(\frac st+\frac1n,X_s^M)ds\right)\right]\right]\\
&\ge\E\left[\exp\left(\int_0^{\frac tn} f_0(X_s^M)ds\right); |X_{\frac tn}^M|<\delta\right] \inf_{|x|<\delta}\E_{x}\left[\exp\left(\int_0^{(1-\frac1n)t} f(\frac st+\frac1n,X_s^M)ds\right)\right],
\end{align*}
where $\E_x$ denotes the expectation with respect to the stable process staring from $x$.

Repeating the above procedure, we can get
\begin{align}\label{e1.5'}
&\prod_{i=0}^{n-1}\inf_{|x|<\delta} \E_x\left[\exp\left(\int_0^{\frac tn} f_i(X_s^M)ds\right);|X_{\frac tn}^M|<\delta\right]
\le \E\left[\exp\left(\int_0^{t} f(\frac st,X_s^M)ds\right)\right].
\end{align}
Similarly,  we have 
\begin{align}\label{e1.6'}
\E\left[\exp\left(\int_0^{t} f(\frac st,X_s^M)ds\right)\right]
\le \prod_{i=0}^{n-1}\sup_{x\in \mathbb T_M^d} \E_x\left[\exp\left(\int_0^{\frac tn} f_i(X_s^M)ds\right)\right].
\end{align}
 
First, we show that 
\begin{equation}\label{lower}
\lim_{t\to\infty}\frac1t \log\inf_{|x|<\delta} \E_x\left[\exp\left(\int_0^{t} f_i(X_s^M)ds\right);|X_{t}^M|<\delta\right]\ge \lambda_M(f_i).
\end{equation}
By boundedness of $f_i$ and the Markov property, we have
\begin{align}
&\E_x\left[\exp\left(\int_0^{t} f_i(X_s^M)ds\right);|X_{t}^M|<\delta\right]\ge C \E_x\left[\exp\left(\int_1^{t-1} f_i(X_s^M)ds\right);|X_{t}^M|<\delta\right]\notag\\
&=C\int_{\mathbb T_M^d}\bar p(y-x)\E_y\left[\exp\left(\int_0^{t-2} f_i(X_s^M)ds\right)\E_{X_{t-2}^M}\left[ I_{[|X_{1}^M|<\delta]}\right]\right] dy, \label{e1.6}
\end{align}
where $\bar p(y)$ is the density function of $X_1^M$. Note that $\bar p(y)$ is strictly positive and continuous on $\mathbb T_M^d$, and Then, there exists $\varepsilon>0$ such that $\inf_{y\in \R^M}\bar p(y)\ge \varepsilon$ and consequently $\inf_{x\in\R^M}\E_{x}\left[ I_{[|X_{1}^M|<\delta]}\right]\ge \varepsilon \delta^d.$ Therefore,
\begin{align} \label{e1.7}
\E_x\left[\exp\left(\int_0^{t} f_i(X_s^M)ds\right);|X_{t}^M|<\delta\right]\ge C\varepsilon^2 \delta^d \int_{\mathbb T_M^d}\E_y\left[\exp\left(\int_0^{t-2} f_i(X_s^M)ds\right)\right] dy. 
\end{align}
On the other hand, for any $g\in \mathcal F_{\alpha, M}$,   
\begin{align}\label{e1.8}
&\int_{\mathbb T_M^d}\E_y\left[\exp\left(\int_0^{t-2} f_i(X_s^M)ds\right)\right] dy\notag\\
&\ge \|g\|_{\infty}^{-2}\int_{\mathbb T_M^d}g(y)\E_y\left[\exp\left(\int_0^{t-2} f_i(X_s^M)ds\right) g(X_{t-2}^M) \right]dy\notag\\
&=\|g\|_{\infty}^{-2} \langle g, e^{-(t-2)(T_{\alpha,M}-V_{f_i})} g\rangle_{2, \mathbb T_M^d},
\end{align}
where  in the last step $T_{\alpha, M}$ is the self-adjoint operator associated with the Dirichlet form $\mathcal E_{\alpha, M}$,  $V_f$ is the operator of the multiplication of the function $f$, and the equality follows from \cite[Lemma 5]{CLR}.  By spectral representation theory, there exists a probability measure $\mu_g(d\lambda)$ such that 
\begin{align}\label{e1.9}
\langle g, f_ig\rangle_{\alpha, M}-\mathcal E_{\alpha,M}(g,g)=\langle g, -(T_{\alpha,M}-V_{f_i})g\rangle_{\alpha, M}=\int_{-\infty}^\infty \lambda \mu_g(d\lambda),
\end{align} and
\begin{align}\label{e1.10}
\langle g, e^{-(t-2)(T_{\alpha,M}-V_{f_i})}g\rangle_{\alpha, M}=\int_{-\infty}^\infty e^{-(t-2)\lambda} \mu_g(d\lambda)\ge \exp\left((t-2)\int_{-\infty}^\infty\lambda \mu_g(d\lambda) \right).
\end{align}
Combining (\ref{e1.9}) and (\ref{e1.10}), we have
\begin{align}\label{e1.11}
\liminf_{t\to\infty}\frac1t\log \langle g, e^{(t-2)(T_{\alpha,M}-V_{f_i})}g\rangle_{\alpha, M}\ge \langle g, f_ig\rangle_{\alpha, M}-\mathcal E_{\alpha,M}(g,g),
\end{align}
and then, by choosing $g$ arbitrarily, (\ref{lower}) follows from  (\ref{e1.7}), (\ref{e1.8}) and (\ref{e1.11}).

Now we show that
\begin{equation}\label{upper}
\limsup_{t\to\infty}\frac1t \log\sup_{x\in \mathbb T_M^d} \E_x\left[\exp\left(\int_0^{t} f_i(X_s^M)ds\right)\right]\le \lambda_M(f_i).
\end{equation}
Actually, by the uniform boundedness of $f_i$ on $\mathbb T_M^d$ and the Markov property of $X^M$,
\begin{align*}
&\E_x\left[\exp\left(\int_0^{t} f_i(X_s^M)ds\right)\right]\le C\E_x\left[\exp\left(\int_1^{t} f_i(X_s^M)ds\right)\right]\\
&\qquad =C\int_{R_M}\bar p(y-x)\E_y\left[\exp\left(\int_0^{t-1} f_i(X_s^M)ds\right)\right] dy\\
&\qquad =C \langle  \bar p, e^{-(t-1)\langle T_{\alpha,M}-V_{f_i})}1\rangle_{2, R_M}.
\end{align*}
By spectral representation, for any $g\in \mathcal F_{\alpha, M}$,
\begin{align*}
\langle g, e^{-(t-1)(T_{\alpha,M}-V_{f_i})}g\rangle_{\alpha, M}=\int_{-\sigma_0}^{\infty} e^{-(t-1)\lambda} \mu_g(d\lambda)\le e^{(t-1)\sigma_0},
\end{align*}
where $-\sigma_0=-\lambda_M(f_i)$ is the infimum  of the spectrum of the operator $T_{\alpha,M}-V_{f_i} $. 
Hence 
\[\limsup_{t\to\infty}\frac1t\log\E_x\left[\exp\left(\int_0^{t} f_i(X_s^M)ds\right)\right]\le \lambda_M(f_i).\]
Combining (\ref{e1.5'}), (\ref{e1.6'}), (\ref{lower}) and (\ref{upper}), we have
\begin{equation}\label{e1.15}
\lim_{t\to\infty}\frac1t \log \E\left[\exp\left(\int_0^{t} f(\frac st,X_s^M)ds\right)\right]=\sum_{i=0}^{n-1}\lambda_M(f_i).
\end{equation}
Finally, for general continuous function $f(s,x)$ on $[0,1]\times \mathbb T_M^d$, let $$f_n(s,x)=\sum_{i=0}^{n-1} f(s_i,x) I_{[s_i, s_{i+1})}(s)+f(s_{n-1},x)I_{\{1\}}(s).$$ Then, by the uniform continuity of $f$ on $[0,1]\times \mathbb T_M^d$, $f_n$ converges to $f$ uniformly. By letting $n$ go to infinity in (\ref{e1.15}), we can obtain  (\ref{fk}). \hfill
\end{proof}
\medskip 

In the meantime, the  lower bound in \eqref{fk} also holds for the original stable process $X$.
\begin{proposition} \label{fklb}For the stable process $X$ on the whole $\R^d,$ if we assume that $f(s,x)$ is continuous in $(s,x)$ on $[0,1]\times \R^d$ and that the family  $\{f(\cdot, x), x\in \R^d\}$ of functions is equicontinuous, Then, we can obtain the lower bound 
\begin{equation}\label{fk-lower}
\liminf_{t\to\infty} \frac1t\log\E\left[\exp\left(\int_0^t f(\frac st, X_s)ds \right)\right]\ge\int_0^1 \lambda(f(s,\cdot))ds,
\end{equation}
where $\lambda(f)=\sup_{g\in\mathcal F_{\alpha}}\Big\{\langle g,fg\rangle_{2,\R^d}-\mathcal E_{\alpha}(g, g)\Big\}.$
\end{proposition} 
\begin{proof} The proof is similar to the lower bound part of  the proof for Proposition \eqref{fk}. We shall 
only  sketch the idea.

We still start with the functions of the form $f(s,x)=\sum_{i=0}^{n-1} f_i(x) I_{[s_i, s_{i+1})}(s)+f_{n-1}(x)I_{\{1\}}(s)$. Fix a compact set $D\subset\R^d$, Then, there exists a positive $\varepsilon$ such that the density function $p(y)$ of $X_1$ is  bigger than $\varepsilon$ for all $y\in D$. For any $g\in \mathcal F_\alpha$ with support inside $D$, using a similar argument as (\ref{e1.7}) --  (\ref{e1.11}), we can get 
\[\liminf_{t\to\infty}\frac1t\log\E_x\left[\exp\left(\int_0^{t} f_i(X_s)ds\right);|X_{t}|<\delta\right]\ge \langle g, f_ig\rangle_{\alpha,\R^d}-\mathcal E_{\alpha}(g,g).\]
Therefore, for any $g\in\mathcal F_\alpha$ with compact support, we have
\begin{equation*}
\liminf_{t\to\infty} \frac1t\log\E\left[\exp\left(\int_0^t f_i(X_s)ds \right)\right]\ge\langle g, f_ig\rangle_{\alpha,\R^d}-\mathcal E_{\alpha}(g,g),
\end{equation*}
and hence
\begin{equation*}
\liminf_{t\to\infty} \frac1t\log\E\left[\exp\left(\int_0^t f_i(X_s)ds \right)\right]\ge\lambda(f_i).
\end{equation*}
Finally, (\ref{fk-lower}) follows from a limiting argument.
 \end{proof}

\section{A variational inequality}
In this section, we will establish a lower bound for $\|u^{\rho}(t,x)\|_p$ for $p\ge1, \rho\in[0,1]$, where $u^\rho$ 
is given by  \eqref{urho} when $\rho\in[0,1)$ under the condition \eqref{Dalang'} and $u^1(t,x)$ is the Skorohod solution $\tilde u(t,x)$ under the condition \eqref{Dalang}. This will be used to obtain the lower bound in Theorem \ref{pa}.

First let us introduce some notations by recalling the Dalang's approach (see \cite{Dalang99}) of defining stochastic integral with respect to the Gaussian noise $\dot W$.   Let $\cD(\R^{d+1})$ be the set of smooth functions on $\R^{d+1}$ with compact support, and $\cH$ be the Hilbert space spanned by $\cD(\R^{d+1})$
 under the inner product 
\begin{equation}
\langle \varphi, \psi\rangle_\cH:=\int_{\R^{2}}\int_{\R^{2d}} |r-s|^{-\beta_0}\gamma(x-y)\varphi(r,x)\psi(s,y) drdsdxdy, \, \forall\, \varphi, \psi\in \cD(\R^{d+1}).
\end{equation}
In the probability space $(\Omega, \cF, \mathbb P)$, let $W=\{W(h), h\in \cH\}$ be an isonormal Gaussian process with covariance function give by $\E[W(h)W(g)]=\langle h, g\rangle_\cH.$
We also write, for $h\in \cH$,
$$W(h)=\int_\R\int_{\R^d} h(s,x) W(ds,dx). $$
Denote  the Fourier transforms of $|s|^{-\beta_0}$ and $\gamma(x)$ by $\mu_0(d\tau)$ and $\mu(d\xi)$, respectively, then
\begin{align}
&\mu_0(d\tau)=C_{\beta_0}|\tau|^{\beta_0-1}d\tau; &\label{ftmu0}\\
&\mu(d\xi)=\begin{cases}\label{ftmu}
 C_{\beta,d}|\xi|^{\beta-d}d\xi,  &\text{ for  }\gamma(x)=|x|^{-\beta}, \\
 \prod_{j=1}^d C_{\beta_j}|\xi|^{\beta_j-1}d\xi,  &\text{ for }  \gamma(x)=\prod_{j=1}^d|x_j|^{-\beta_j}.
 \end{cases}
 \end{align}
 The Parseval's identity provides an alternative representation for the inner product, $$\E[W(\varphi)W(\psi)]=\langle \varphi, \psi\rangle_\cH=\int_\R\int_{\R^d}\varphi(\tau, \xi)\overline{\widehat \psi(\tau,\xi)}\mu_0(d\tau)\mu(d\xi), \text{ for } \varphi, \psi\in \cS(\R^{d+1}).$$

 With the above notations    \eqref{Dalang} is equivalent to the following general
 form of the Dalang's condition
 \begin{equation}\label{Dalang-add}
\int_{\R^d} \frac1{1+|\xi|^{\alpha}} \mu(d\xi)<\infty,
\end{equation} 
and   \eqref{Dalang'} is equivalent to 
\begin{equation}\label{Dalang'general-add}
\int_{\R^d} \frac{1}{1+|\xi|^{\alpha(1-\beta_0)}}\mu(d\xi)<\infty\,. 
\end{equation}
 
Now we recall the approximation procedure used in \cite{hn09, hns, Song}, which we shall use in the proof of the main result in this section. Denote $g_\delta(t):=\frac1\delta I_{[0,\delta]}(t)$ for $t\ge 0$ and $p_\varepsilon(x)=\frac1{\varepsilon^{d}}p(\frac{x}{\varepsilon})$ for $x\in \R^d$, 
where $p(x)\in \cD(\R^d)$ is a symmetric probability 
density function and its Fourier transform $\widehat p(\xi)\ge 0$  for all $\xi\in\R^d.$  
For positive numbers $\varepsilon$ and $\delta$, define
\begin{align} \label{apprw}
\dot{W}^{\epsilon ,\delta
}(t,x):=\int_{0}^{t}\int_{\mathbb{R}^{d}}g_{\delta
}(t-s)p_{\epsilon }(x-y)W(ds,dy)=W(\phi_{t,x}^{\varepsilon, \delta}),
\end{align}
where 
$$\phi_{t,x}^{\e,\delta}(s,y):=g_\delta(t-s)p_\e(x-y) \cdot I_{[0,t]}(s).$$
Consider the following approximation of (\ref{pam})
\begin{equation}\label{appspde}
\begin{cases}
u^{\varepsilon, \delta}(t,x)=-(-\Delta)^{\frac \alpha2} u^{\varepsilon, \delta}(t,x)+u^{\varepsilon, \delta}(t,x)\dot W^{\varepsilon, \delta}(t,x),\\
u^{\varepsilon,\delta}(0,x)=u_0(x).
\end{cases}
\end{equation}
Then, Feynman-Kac formula for the Stratonovich solution $u^{\varepsilon, \delta}$ is
\[u^{\varepsilon, \delta}(t,x)=\E_X\left[u_0(X_t^x)\exp\left(\int_0^t \dot W^{\varepsilon, \delta}(r,X_{t-r}^x)dr \right)\right],\] 
and the Feynman-Kac formula for the Skorohod solution $\tilde u^{\e,\de}(t,x)$ is 
\begin{align*}
 \tilde u^{\varepsilon, \delta}(t,x)=&\E_X\Bigg[u_0(X_t^x)\exp\bigg(\int_0^t \dot W^{\varepsilon, \delta}(r,X_{t-r}^x)dr-\frac12\int_{\R^{d+1}} |\mathcal F\Phi_{t,x}^{\varepsilon,\delta}(\tau,\xi)|^2 \mu_0(d\tau)\mu(d\xi)\bigg) \Bigg],
 \end{align*}
where  
\begin{equation}\label{Phit}
\Phi_{t,x}^{\varepsilon,\delta}(u,y):=\int_0^{t} g_\delta(t-u-s)p_\varepsilon(X_s^x-y)ds\cdot I_{[0,t]}(u).
\end{equation}
Notet that
$$  \int_0^t \dot W^{\varepsilon, \delta}(r,X_{t-r}^x)dr =\int_{\R}\int_{\R^d}\Phi_{t,x}^{\e,\delta}(u,y)W(du,dy), $$
by stochastic Fubini's theorem.

For $ \rho\in [0,1]$, define the following random Hamiltonian, 
$$H_{\e,\de}^{\rho}(t,x):=
\int_0^t \dot W^{\varepsilon, \delta}(r,X_{t-r}^x)dr -\frac\rho2 \int_{\R^{d+1}} |\mathcal F\Phi_t^{\varepsilon,\delta}(\tau,\xi)|^2 \mu_0(d\tau)\mu(d\xi), $$
and denote
\begin{equation}\label{urho'}
u_{\e,\de}^{\rho}(t,x):=\E_X\left[\exp\left(H_{\e,\de}^{\rho}(t,x)\right)\right].
\end{equation}
Then, for all fixed $(t,x)\in \R_+\times\R^{d}$,  under the condition \eqref{Dalang'},  for all $\rho\in[0,1]$, $H_{\e,\de}^{\rho}(t,x)$ converges to $H^\rho(t,x)$ given in \eqref{hrho} (see Theorem 4.1 in \cite{Song})
and  $u^\rho_{\e,\de}(t,x)$ converges to $ u^{\rho}(t,x):=\E_X\left[\exp\left(H^{\rho}(t,x)\right)\right]$  in $L^p$ for all $p\ge1$ (see Theorem 4.6 in \cite{Song}). Under the less restricted condition \eqref{Dalang}, when $\rho=1$, $u^1_{\e,\de}(t,x)$ converges to  the Skorohod solution $\tilde u(t,x)$ of \eqref{pam} in $L^p$ for all $p\ge1$ (see Theorem 5.6 in \cite{Song}).
  
\bigskip

The following is the main result in this section.
\begin{proposition}\label{prop3.1}
We assume  one of the following conditions
\begin{enumerate}
\item[(i)]  The condition \eqref{Dalang'} is  satisfied and $\rho\in [0,1]$. 
\item[(ii)]   Dalang's condition \eqref{Dalang}  is satisfied and    $\rho=1$. 
\end{enumerate}
Let $p\ge 1$, and when $p=1$ we assume $\rho\in [0,1).$ Then, for any $(t,x)\in \R^+\times \R^d$, 
\begin{align*}
&\left(\E|u^{\rho}(t,x)|^p\right)^{1/p}\\
&\ge \sup_{g\in \cS_H(\R^{d+1})} \E_X \left[\exp\left(\int_0^t (\widetilde{\cF} g)(s,X_s)ds -\frac{1}{2(p-\rho)}\int_{\R^{d+1}} |g(\tau,\xi)|^2 \mu_0(d\tau) \mu(d\xi)\right)\right],
\end{align*}
where
\begin{equation}\label{schwartz'}
\cS_H(\R^{d+1})=\Big\{g\in \cS(\R^{d+1}); \, g(-\tau, -\xi)=\overline{g(\tau,\xi)}\Big\},
\end{equation}
and 
 \begin{equation}\label{fourier'}
 (\widetilde{\cF} g)(s,x)=\int_{\R^{d+1}} e^{-2\pi i(\tau s+\xi\cdot x)} g(\tau, \xi) \mu_0(d\tau) \mu(d\xi).
 \end{equation}
\end{proposition}
\begin{proof}
First, we consider the case $p>1$ and $\rho\in[0,1]$. Let $q:=p(p-1)^{-1}$ be the conjugate of $p.$
Let $\varphi(t,x)\in \cS(\R^{d+1})$ be a real function, and denote
$$X_{\varphi}=\exp\left(\int_\R \int_{\R^d} \varphi(s,y)W(ds,dy)-\frac q{2}\int_{\R^{d+1}}|\widehat \varphi (\tau, \xi)|^2 \mu_0(d\tau)\mu(d\xi)\right).$$
Note that $X_\varphi\in L^q(\Omega)$ and $\|X_\varphi\|_q=1.$
Hence, by   H\"older's inequality, we see 
\begin{align*}
&\|u^\rho_{\e,\delta}(t,x)\|_p\ge \E\left[  u^\rho_{\e,\delta}(t,x) X_\varphi \right]\\
&=\E_W\E_X\Bigg[\exp\bigg(\int_\R\int_{\R^d} \left[\Phi_{t,x}^{\e,\delta}(s,y)+\varphi(s,y)\right]W(ds,dy)\\
&~~~~-\frac\rho2 \int_{\R^{d+1}} |\mathcal F\Phi_{t,x}^{\varepsilon,\delta}(\tau,\xi)|^2 \mu_0(d\tau)\mu(d\xi) - \frac q2 \int_{\R^{d+1}} | \widehat\varphi(\tau,\xi)|^2 \mu_0(d\tau)\mu(d\xi) \bigg) \Bigg]\\
&=\E_X \Bigg[\exp\bigg( \frac{1-\rho}2 \int_{\R^{d+1}} |\mathcal F\Phi_{t,x}^{\varepsilon,\delta}(\tau,\xi)|^2 \mu_0(d\tau)\mu(d\xi) \\
&  ~~~~+\int_{\R^{d+1}}\overline{ \mathcal F\Phi_{t,x}^{\varepsilon,\delta}(\tau,\xi)} \widehat \varphi(\tau,\xi)\mu_0(d\tau)\mu(d\xi)  - \frac{q-1}{2}\int_{\R^{d+1}} | \widehat\varphi(\tau,\xi)|^2 \mu_0(d\tau)\mu(d\xi) \bigg) \Bigg].
\end{align*}
Note that for any $x\ge1$, 
\begin{eqnarray*}
&&(1-\rho)a^2+2ab-(q-1)b^2
=(1-\rho)a^2+2(1-x)ab+2xab-(q-1)b^2\\
&&\qquad \ge  -\frac{(x-1)^2}{1-\rho} b^2+2xab-(q-1)b^2 
=2xab-\left((q-1)+\frac{(x-1)^2}{1-\rho}\right)b^2\,.
\end{eqnarray*}
 If we choose the optimal value $c_0=1+(1-\rho)(q-1)$ for $x$, Then, we have 
 \[
 (1-\rho)a^2+2ab-(q-1)b^2\ge 2a(c_0b)- \frac{1}{p-\rho}(c_0b)^2\,.
 \]
 This argument also works with the product $ab$ replaced by  inner products $\langle \cdot, \cdot \rangle_\mathcal H$, noting that $\int_{\R^{d+1}}\overline{ \mathcal F\Phi_{t,x}^{\varepsilon,\delta}(\tau,\xi)} \widehat \varphi(\tau,\xi)\mu_0(d\tau)\mu(d\xi)$ is a real (random) number. Therefore,
\begin{align}
\|u_{\e,\delta}^\rho(t,0)\|_p
\ge&\E_X \Bigg[\exp\bigg( \int_{\R^{d+1}} \overline{\mathcal F\Phi_{t,x}^{\varepsilon,\delta}(\tau,\xi) } \Big( c_0\widehat \varphi(\tau,\xi)\Big)\mu_0(d\tau)\mu(d\xi) \notag\\
& - \frac12 \frac1{p-\rho}\int_{\R^{d+1}} | c_0\widehat \varphi(\tau,\xi)|^2 \mu_0(d\tau)\mu(d\xi) \bigg) \Bigg]. \label{e3.1}
\end{align}
Note that 
\begin{align*}
\mathcal F\Phi_{t,x}^{\varepsilon,\delta}(\tau,\xi) =\int_0^t \exp(-2\pi i(\tau(t-s)+\xi\cdot X_s)) \cF \left(\frac{1}{\delta}I_{[0, (t-s)\wedge \delta]}(\cdot)\right)(\tau) \widehat p_\e(\xi) ds
\end{align*}
which converges to $\int_0^t \exp(-2\pi i(\tau(t-s)+\xi\cdot X_s)) ds$ as $ \e$ and $\delta $ go  to $0.$
Letting $ \e$ and $\delta $ go to $0$ in \eqref{e3.1} yields 
\begin{align*}
\|u^\rho(t,0)\|_p
\ge&\E_X \Bigg[\exp\bigg( \int_0^t \int_{\R^{d+1}} \exp(-2\pi i(\tau(t-s)+\xi\cdot X_s)) \Big( c_0\widehat \varphi(\tau,\xi)\Big)\mu_0(d\tau)\mu(d\xi) ds\notag\\
& - \frac12 \frac1{p-\rho}\int_{\R^{d+1}} | c_0\widehat \varphi(\tau,\xi)|^2 \mu_0(d\tau)\mu(d\xi) \bigg) \Bigg]. 
\end{align*}
The proof is concluded for the case $p>1$, noting that $\cF(\cS(\R^{d+1}))= \cS(\R^{d+1})$, and $\widehat \varphi(-\tau,-\xi)=\overline{\widehat\varphi(\tau,\xi)}$ since $\varphi$ is a real function.

When $p=1$ and $\rho\in[0,1)$,   we have 
\begin{align*}
\E[u_{\e,\de}^\rho(t,x)]=&\E_X \Bigg[\exp\bigg( \frac{1-\rho}2 \int_{\R^{d+1}} |\mathcal F\Phi_{t,x}^{\varepsilon,\delta}(\tau,\xi)|^2 \mu_0(d\tau)\mu(d\xi) \bigg)\Bigg]\\
\ge&\E_X \Bigg[\exp\bigg( \int_{\R^{d+1}} \overline{\mathcal F\Phi_{t,x}^{\varepsilon,\delta}(\tau,\xi) } \Big( c_0\widehat \varphi(\tau,\xi)\Big)\mu_0(d\tau)\mu(d\xi) \notag\\
& - \frac12 \frac1{1-\rho}\int_{\R^{d+1}} | c_0\widehat \varphi(\tau,\xi)|^2 \mu_0(d\tau)\mu(d\xi) \bigg) \Bigg]. 
\end{align*}
where the last step follows from  $(1-\rho)a^2\ge 2ab-\frac1{1-\rho}b^2$.
The result can be deduced  in a similar way.  
\end{proof}

\begin{remark}\rm{The result still holds if the $\alpha$-stable process $X$ in  $u^\rho(t,x)$ is replaced by a general symmetric L\'evy process with characteristic function $\E[e^{i\xi\cdot X_t}]=e^{-t\Psi(\xi)}$. In this case, the conditions \eqref{Dalang'} and \eqref{Dalang} are $\int_{\R^d} \frac1{1+[\Psi(\xi)]^{1-\beta_0}}\mu(d\xi)<\infty$ and $\int_{\R^d} \frac1{1+\Psi(\xi)}\mu(d\xi)<\infty$, respectively.}
\end{remark}

\section{On the lower bound}

 In this section, we establish the lower bound in Theorem \ref{pa} for all $p\ge1.$ 

Note that  $\mu_0(d(c\tau))=c^{\beta_0}\mu_0(d\tau)$ and $\mu(d(c\xi))=c^\beta\mu(d\xi)$ for any $c>0$, by \eqref{ftmu0} and \eqref{ftmu}. Consequently, for  $h\in \cS_H(\R^{d+1})$,   where $\cS_H(\R^{d+1})$ is given in \eqref{schwartz'},  we have 
\begin{equation}\label{e4.1}
(\widetilde{\cF} h(a\cdot, b*))(s,x)=a^{-\beta_0}b^{-\beta}(\widetilde{\cF} h(\cdot, *))(a^{-1}s, b^{-1}x), \, a>0, b>0,
\end{equation}
where $\widetilde{\cF} g$ is defined by \eqref{fourier'}.

Now let  
\begin{equation}\label{tp}
t_p=t^{\chi}(p-\rho)^{\frac{\alpha}{\alpha-\beta}}  
\text{ for } p\ge 1, \text{ with } \chi=\frac{2\alpha-\beta-\alpha\beta_0}{\alpha-\beta},
\end{equation}
and for any $h\in\cS_H(\R^{d+1})$ denote 
$$h_t(\tau,\xi)=t (p-\rho) h\Big(t\tau, (p-\rho)^{-\frac{1}{\alpha-\beta}} t^{-\frac{\chi-1}{\alpha}}\xi\Big).$$
Then, by \eqref{e4.1}, change of variables and the self-similarity of the $\alpha$-stable process, we have
$$\int_0^{t_p} (\widetilde \cF h)(\frac{s}{t_p}, X_s)ds \overset{d}{=}\int_0^t (\widetilde \cF h_t)(s,X_s)ds,$$
and $$ \int_{\R^{d+1}} |h_t(\tau,\xi)|^2 \mu_0(d\tau) \mu(d\xi)=(p-\rho) t_p \int_{\R^{d+1}} |h(\tau,\xi)|^2 \mu_0(d\tau) \mu(d\xi).$$
Clearly,  $h_t\in \cS_H(\R^{d+1})$.   Proposition \ref{prop3.1} and the above two identities imply 
\begin{align*}
\|u^{\rho}(t,x)\|_p&\ge  \E_X \left[\exp\left(\int_0^t (\widetilde{\cF} h_t)(s,X_s)ds -\frac{1}{2(p-\rho)}\int_{\R^{d+1}} |h_t(\tau,\xi)|^2 \mu_0(d\tau) \mu(d\xi)\right)\right]\\
&=  \E_X \left[\exp\left(\int_0^{t_p} (\widetilde{\cF} h)(\frac{s}{t_p},X_s)ds -\frac{t_p}{2}\int_{\R^{d+1}} |h(\tau,\xi)|^2 \mu_0(d\tau) \mu(d\xi)\right)\right].
\end{align*}
By Proposition \ref{fklb}, 
\begin{align*}
&\liminf_{t\to\infty}\frac1{t_p} \log \E_X \left[\exp\left(\int_0^{t_p} (\widetilde{\cF} h)(\frac{s}{t_p},X_s)ds \right)\right] \ge \int_0^1 \lambda((\widetilde{\cF} h)(s,\cdot))ds\\
&=\int_0^1\sup_{g\in {\cal F}_\alpha}\bigg\{
\int_{\R^d}(\widetilde{\cF} h)(s,x)g^2(x)dx-
\int_{\R^d} |\xi|^\alpha |\widehat g(\xi)|^2d\xi \bigg\}~ds\\
&=\sup_{g\in {\cal A}_{\alpha,d}}\bigg\{\int_0^1\!\!\int_{\R^d}(\widetilde{\cF} h)(s,x)g^2(s, x)dxds-
\int_0^1\!\!\int_{\R^d} |\xi|^\alpha |\widehat g(s,\xi)|^2d\xi ds\bigg\}\,,
\end{align*}
where ${\cal A}_{\alpha,d}$ is given by \eqref{Ad}. 
Therefore, 
\begin{align}
&\liminf_{t\to\infty} t^{-\chi} \log\|u^{\rho}(t,x)\|_p\notag\\
&\ge (p-\rho)^{\frac{\alpha}{\alpha-\beta}} \sup_{g\in {\cal A}_{\alpha,d}}\bigg\{\Gamma(h,g)-
\int_0^1\!\!\int_{\R^d} |\xi|^\alpha |\widehat g(s,\xi)|^2d\xi ds\bigg\}\notag\\
&\ge (p-\rho)^{\frac{\alpha}{\alpha-\beta}} \sup_{g\in {\cal A}_{\alpha,d}}\bigg\{\sup_{h\in\cS_H(\R^{d+1})}\Gamma(h,g)-
\int_0^1\!\!\int_{\R^d} |\xi|^\alpha |\widehat g(s,\xi)|^2d\xi ds\bigg\}, \label{e4.2}
\end{align}
where 
\begin{align*}
\Gamma(h,g)&=\int_0^1\!\!\int_{\R^d}(\widetilde{\cF} h)(s,x)g^2(s, x)dxds-\frac{1}{2}\int_{\R^{d+1}} |h(\tau,\xi)|^2 \mu_0(d\tau) \mu(d\xi)\\
&=\int_{\R^{d+1}}h(\tau,\xi)(\cF g^2)(\tau, \xi)\mu_0(d\tau)\mu( d\xi)-\frac{1}{2}\int_{\R^{d+1}} |h(\tau,\xi)|^2 \mu_0(d\tau) \mu(d\xi).
\end{align*}
Since $\cS_H(\R^{d+1})$ is dense in $L^2(\R^{d+1}, \mu_0\otimes\mu)$ (see, e.g., \cite{Jolis10}), and $\Gamma(\cdot, g)$ is continuous with respect to the $L^2(\R^{d+1}, \mu_0\otimes \mu)$-norm,  we have
\begin{align*}
&\sup_{h\in\cS_H(\R^{d+1})}\Gamma(h,g)\ge \Gamma\Big(\cF(g^2)(-\tau, -\xi), g\Big)=\frac12\int_{\R^{d+1}}|(\cF g^2)(\tau, \xi)|^2\mu_0(d\tau) \mu(d\xi)\\
& =\frac12\int_0^1\int_0^1 \int_{\R^{2d}} \frac{\gamma(x-y)}{|s-r|^{\beta_0}} g^2(s,x)g^2(r,y) dxdydrds.
\end{align*}
Summarizing the computations starting from \eqref{e4.2}, we have
\begin{align*}
&\liminf_{t\to\infty} t^{-\chi} \log\|u^{\rho}(t,x)\|_p\\
&\ge (p-\rho)^{\frac{\alpha}{\alpha-\beta}} \sup_{g\in {\cal A}_{\alpha,d}}\bigg\{\frac12\int_0^1\int_0^1 \int_{\R^{2d}} \frac{\gamma(x-y)}{|s-r|^{\beta_0}} g^2(s,x)g^2(r,y) dxdydrds\\
&\qquad -
\int_0^1\!\!\int_{\R^d} |\xi|^\alpha |\widehat g(s,\xi)|^2d\xi ds\bigg\}\\
&= (p-\rho)^{\frac{\alpha}{\alpha-\beta}} \mathbf M(\alpha, \beta_0, d,\gamma),
\end{align*}
and the lower bound is established.

\section{On the upper bound}
In this section, we provide a proof for  the upper bound  in Theorem \ref{pa}. In Subsections \ref{section6.1} and \ref{section6.2}, we shall obtain the upper bound for any positive integer $n\ge 1$, i.e.,
\begin{align}
&\limsup_{t\to\infty}t^{-\frac{2\alpha-\beta-\alpha\beta_0}{\alpha-\beta}}\log\E\exp\Bigg(\frac12\sum_{j,k=1}^n \int_0^t\int_0^t |r-s|^{-\beta_0}\gamma(X_r^j-X_s^k)drds\notag\\
& \qquad\qquad\qquad\qquad\qquad-\frac\rho2\sum_{j=1}^n\int_0^t\int_0^t |r-s|^{-\beta_0}\gamma(X_r^j-X_s^j)drds\Bigg)\notag\\
&\qquad \qquad \le n(n-\rho)^{\frac{\alpha}{\alpha-\beta}}\mathbf M(\alpha, \beta_0, d,\gamma). \label{eq6.1'}
\end{align}
The proof  for real number $p\ge 2$ is inspired    by the idea in \cite{Le}.  We shall
compare $\|u^\rho(t,x)\|_p$ with $\|u^\rho(t,x)\|_2$ by  using the Mehler's formula and hypercontractivity of the Ornstein-Uhlenbeck semigroup operators.     First,  we address the case when $\rho\in[0,1]$, under the condition \eqref{Dalang'}.

Let $W'=\{W'(h), h\in \cH\}$ be an independent copy of $W=\{W(h), h\in\cH\}$, and let $W: \Omega\to \R^\cH$ and $ W':\Omega\to \R^\cH$ be the canonical mappings associated with $W$ and $W'$, respectively. For any  $F\in L^2(\Omega)$,  there is a     measurable mapping 
$\psi_F$  from $\R^\cH$ to $\R$ such that    $F=\psi_F\circ W$. Denote by $\{T_\tau,\tau\ge 0\}$ the Ornstein-Uhlenbeck semigroup associated with $W$.  By Mehler's formula (see, e.g., \cite{Nualart06}), 
$$T_\tau(F)=\E'\left[\psi_F(e^{-\tau} W+\sqrt{1-e^{-2\tau}}W')\right],$$
where $\E'$ denotes the expectation with respect to $W'$. For  $p\in (1, \infty)$ and $\tau\ge 0$, define $q=1+e^{2\tau}(p-1),$  Then, the Ornstein-Uhlenbeck semigroup operators possess the following hypercontractivity property (see, e.g., \cite{Nualart06}),
\begin{equation}\label{hyper}
\|T_\tau F\|_q\le \|F\|_p.
\end{equation}
Now fix $q\ge2$. Let $e^{2\tau}=q-1$, Then, $\|T_\tau F\|_q\le \|F\|_2$. Let $\tilde \rho=\frac{\rho+q-2}{q-1}\in[0,1).$ By \eqref{urho} and Mehler's formula, 
\begin{align*}
&T_\tau u^{\tilde \rho}(t,x)= \E'\E_X\Bigg[ \exp \bigg(
e^{-\tau}\int_{0}^{t}\int_{\mathbb{R}^d} \delta_0
(X_{t-r}^{x}-y)W(dr,dy)\notag\\
&~~~+\sqrt{1-e^{-2\tau}}\int_{0}^{t}\int_{\mathbb{R}^d} \delta_0
(X_{t-r}^{x}-y)W'(dr,dy) -\frac{\tilde \rho}2\int_0^t\int_0^t |r-s|^{-\beta_0}\gamma(X_r-X_s)drds\bigg)\Bigg]\\
&=\E_X\Bigg[ \exp \bigg(
e^{-\tau}\int_{0}^{t}\int_{\mathbb{R}^d} \delta_0
(X_{t-r}^{x}-y)W(dr,dy)\notag\\
&~~~\qquad\qquad +\frac12(1-\tilde \rho-e^{-2\tau})\int_0^t\int_0^t |r-s|^{-\beta_0}\gamma(X_r-X_s)drds\bigg)\Bigg]\\
&=\E_X\Bigg[ \exp \bigg(
\int_{0}^{t}\int_{\mathbb{R}^d} \delta_0
(X_{t-r}^{x}-y)W_\tau(dr,dy)\notag\\
&~~~\qquad\qquad -\frac\rho2\int_0^t\int_0^t |r-s|^{-\beta_0}\gamma_\tau(X_r-X_s)drds\bigg)\Bigg],
\end{align*}
where in the last step $W_\tau=e^{-\tau}W$ and $\gamma_\tau(x)=e^{-2\tau}\gamma(x)$.
By \eqref{hyper} with $p=2$, \eqref{eq6.1'} with $n=2$, and the scaling property for $\mathbf M(\alpha, \beta_0, d,\gamma)$ defined by  \eqref{mscaling}, we have
\begin{align*}
&\|T_\tau u^{\tilde \rho}(t,x)\|_q\le (2-\tilde\rho)^{\frac{\alpha}{\alpha-\beta}}\mathbf M(\alpha, \beta_0, d, \gamma)\\
&= (2-\tilde\rho)^{\frac{\alpha}{\alpha-\beta}}e^{\frac{2\tau\alpha}{\alpha-\beta}}\mathbf M(\alpha, \beta_0, d, \gamma_\tau) =(q-\rho)^{\frac{\alpha}{\alpha-\beta}}\mathbf M(\alpha, \beta_0, d, \gamma_\tau).
\end{align*}
Observing that $$T_\tau u^{\tilde \rho}(t,x)=\E_X\Bigg[ \exp \bigg(
\int_{0}^{t}\int_{\mathbb{R}^d} \delta_0
(X_{t-r}^{x}-y)W_\tau(dr,dy)-\frac\rho2\int_0^t\int_0^t |r-s|^{-\beta_0}\gamma_\tau(X_r-X_s)drds\bigg)\Bigg],$$
the upper bound in Theorem \ref{pa}  for any real number $q\ge2$ follows from the scaling property \eqref{mscaling}.

 Finally, for the case $\rho=1$ under the condition \eqref{Dalang}, in which $u^\rho(t,x)$ is the Skorohod solution to \eqref{pam},  we can apply the approach in \cite{Le} and obtain the upper bound for all real numbers $p\ge2$.

\subsection{Upper bound under the condition \eqref{Dalang'}.} \label{section6.1}

In this subsection, we deal with the case $\rho\in [0,1]$ under the condition \eqref{Dalang'}. The proof will be split into four steps. 

\noindent{\bf Step 1.} In this step, we will reduce the study 
of $n$-th moment to  the study of  first moment.
Recall that \eqref{timekernel} and \eqref{spacekernel} imply
\begin{align}\label{transform-hamiltonian}
&\int_0^t\int_0^t |r-s|^{-\beta_0}\gamma(X_r^j-X_s^k)drds\notag\\
&=C_0C(\gamma)\int_{\R^{d+1}} \left(\int_0^t |s-u|^{-\frac{\beta_0+1}{2}}K(x-X_s^k) ds\int_0^t |r-u|^{-\frac{\beta_0+1}{2}}K(x-X_r^j) dr\right) dudx\,. 
\end{align}
Therefore, by the inequality  $(\sum_{j=1}^n a_j)^2\le n\sum_{j=1}^n a_j^2$, we have
\begin{align*}
&\sum_{j,k=1}^n \int_0^t\int_0^t |r-s|^{-\beta_0}\gamma(X_r^j-X_s^k)drds-\rho\sum_{j=1}^n\int_0^t\int_0^t |r-s|^{-\beta_0}\gamma(X_r^j-X_s^j)drds\\
&\le (n-\rho)\sum_{j=1}^n \int_0^t\int_0^t |r-s|^{-\beta_0}\gamma(X_r^j-X_s^j)drds.
\end{align*}
Consequently, to obtain the upper bound in Theorem \ref{pa}, it suffices to show
\begin{align}
&\limsup_{t\to\infty}t^{-\frac{2\alpha-\beta-\alpha\beta_0}{\alpha-\beta}}\log\E\left[\exp\left(\frac{n-\rho}2\sum_{j=1}^n \int_0^t\int_0^t |r-s|^{-\beta_0}\gamma(X_r^j-X_s^j)drds\right)\right]\notag\\
&\le n (n-\rho)^{\frac{\alpha}{\alpha-\beta}}{\bf M}(\alpha,\beta_0, d,\gamma). \label{eq5.5}
\end{align}
By the scaling property \eqref{scaling}, we see 
$$\int_0^t\int_0^t |r-s|^{-\beta_0}\gamma(X_r^j-X_s^j)drds\overset{d}{=}\frac1{t_n}\frac{1}{n-\rho}\int_0^{t_n}\int_0^{t_n} \frac{\gamma(X_r^j-X_s^j)}{|t_n^{-1}(r-s)|^{\beta_0}}drds,$$
where $t_n=t^{\frac{2\alpha-\beta-\alpha\beta_0}{\alpha-\beta}}(n-\rho)^{\frac{\alpha}{\alpha-\beta}}$ is given in \eqref{tp}. Therefore, \eqref{eq5.5}  is equivalent to 
\begin{equation}\label{upperbound}
\limsup_{t\to\infty} \frac1t \log \E\left[\exp\left(\frac1{2t} \int_0^{t}\int_0^{t} \frac{\gamma(X_r-X_s)}{|t^{-1}(r-s)|^{\beta_0}}drds\right)\right]\le {\bf M}(\alpha, \beta_0, d,\gamma).
\end{equation}
Now, to obtain the upper bound, it suffices to prove \eqref{upperbound}. To this goal,  we shall
use the representations \eqref{timekernel}  and \eqref{spacekernel}  for the covariance functions.
But in these two representations, the integrals are over  infinite domains. We shall approximate 
them  by bounded, continuous,  and locally supported functions, and this will enable us to apply Hahn-Banach theorem in Step 4.   


{\bf Step 2.} In this step, we will replace the temporal covariance function by a smooth function with compact support.    Let the function $\varrho:\R^+\to[0,1]$ be a smooth function such that 
$\varrho(u)=1, u\in [0,1]$, $\varrho(u)=0$ for $u\ge2$, and $-1\le \varrho'(u)\le 0.$ Define the following truncated functions 
 \begin{equation}\label{kaa}
 k_A(u)=|u|^{-\frac{1+\beta_0}{2}}\varrho(A^{-1}|u|),\, k_{A,a}(u)= |u|^{-\frac{1+\beta_0}{2}}\varrho(A^{-1}|u|)(1-\varrho(a^{-1}|u|)),
 \end{equation}
with $A>0$ being a large number and $a>0$ being a number close to zero.

Then, by H\"older's inequality, we have for any $\e>0$
\begin{align}
&\E\left[\exp\left(\frac1{2t} \int_0^{t}\int_0^{t} \frac{\gamma(X_r-X_s)}{|t^{-1}(r-s)|^{\beta_0}}drds\right)\right]\notag\\
=&\E\left[\exp\left(C_0C(\gamma)\frac1{2t}\int_{\R^{d+1}} \left(\int_0^t |t^{-1}(s-u)|^{-\frac{\beta_0+1}{2}}K(x-X_s) ds\right)^2 dudx\right)\right]\notag\\
\le& \left(\E \left[\exp\left((1+\e)C_0C(\gamma)\frac p{2t}\int_{\R^{d+1}} \left(\int_0^t k_{A,a}(t^{-1}(s-u))K(x-X_s) ds\right)^2 dudx\right)\right]\right)^{1/p}\notag\\
&\times \left(\E \left[\exp\left((1+\frac1\e)C_0C(\gamma)\frac q{2t}\int_{\R^{d+1}} \left(\int_0^t \tilde k_{A,a}(t^{-1}(s-u))K(x-X_s) ds\right)^2 dudx\right)\right]\right)^{1/q},\label{eq6.7}
\end{align}
where $$\tilde k_{A,a}(u)=|u|^{-\frac{1+\beta_0}{2}}-k_{A,a}(u).$$
Note that 
\begin{align}
&\tilde k_{A,a}(u)=(|u|^{-\frac{1+\beta_0}{2}}-k_A(u))+(k_A(u)-k_{A,a}(u))\notag\\
&\le |u|^{-\frac{1+\beta_0}{2}} I_{[|u|\ge A]}+|u|^{-\frac{1+\beta_0}{2}}I_{[|u|\le 2a]}\notag\\
&\le A^{-\frac{\beta_0-\beta_0'}{2}}|u|^{-\frac{\beta_0'+1}{2}}+(2a)^{\frac{\tilde \beta_0-\beta_0}{2}}|u|^{-\frac{\tilde\beta_0+1}{2}}, \label{eq6.8}
\end{align}
for $0<\beta_0'<\beta_0<\tilde \beta_0<1$. We may choose $\beta_0'$ and $\tilde \beta_0$ such that  $(\alpha, \beta_0', \beta)$ and $(\alpha, \tilde \beta_0, \beta)$ satisfy the condition \eqref{Dalang'} if $\rho\in [0,1)$ or the  condition \eqref{Dalang} if $\rho=1$.

 Combining \eqref{timekernel} and \eqref{eq6.8}, for the second term in \eqref{eq6.7},  we have
\begin{align}
&\limsup_{t\to\infty}\frac1t\log\E \left[\exp\left((1+\frac1\e)C_0C(\gamma)\frac q{2t}\int_{\R^{d+1}} \left(\int_0^t \tilde k_{A,a}(t^{-1}(s-u))K(x-X_s) ds\right)^2 dudx\right)\right]\notag\\
&\le \limsup_{t\to\infty}\frac1t\log\E\left[\exp\left(C(\e, q) \left[A^{-(\beta_0-\beta_0')}\frac1{2t}\int_0^t\int_0^t |r-s|^{-\beta_0'}\gamma(X_r-X_s)drds\right.\right.\right.\notag\\
&\quad \quad \quad \quad \quad \quad \quad \quad \quad \left.\left.\left.+(2a)^{\tilde\beta_0-\beta_0}\frac1{2t}\int_0^t\int_0^t |r-s|^{-\tilde\beta_0}\gamma(X_r-X_s)drds\right]\right)\right]\notag\\
&\le C\Big(\alpha, \beta, \e,q, \gamma(\cdot)\Big)  \left(A^{-\frac{\alpha(\beta_0-\beta_0')}{\alpha-\beta}}+(2a)^{\frac{\alpha(\tilde\beta_0-\beta_0)}{\alpha-\beta}}\right)\label{eq6.9'}
\end{align}
where the last step follows from   H\"older's inequality and  \eqref{bound-3'}. Therefore, for fixed $(\e, q)$, this term can be as small as we  wish  if we choose $A$ sufficiently large and $a$ sufficiently small.
On the other hand, we can choose $\e$ arbitrarily close to 0 and $p$ arbitrarily close to 1. Consequently, to prove \eqref{upperbound}, it suffices to prove 
\begin{align}
&\limsup_{t\to \infty}\frac1t\E \left[\exp\left(C_0C(\gamma)\frac 1{2t}\int_{\R^{d+1}} \left(\int_0^t k_{A,a}(t^{-1}(s-u))K(x-X_s) ds\right)^2 dudx\right)\right]\notag\\
&\le \mathbf M(\alpha, \beta_0, d, \gamma). \label{eq6.9}
\end{align}

{\bf Step 3.} In this step, we will replace the spatial covariance function by a smooth function with compact support.
Similarly to the truncation for the temporal covariance function, for $0<b<B<\infty$, we let 
 $$K_{B,b}(x)= K(x)\varrho(B^{-1}|x|)(1-\varrho(b^{-1}|x|)),$$
 where $K(x)$ is given in \eqref{bound-9}. Then, $0\le K_{B,b}(x)\le K(x)$ and $K_{B,b}(x)\to K(x)$ when $B\to\infty$ and $b\to 0.$ Now the left-hand side of \eqref{eq6.9} can be estimated 
 in the  similar way  as in \eqref{eq6.7}, i.e., 
 \begin{align}
&\E \left[\exp\left(C_0C(\gamma)\frac 1{2t}\int_{\R^{d+1}} \left(\int_0^t k_{A,a}(t^{-1}(s-u))K(x-X_s) ds\right)^2 dudx\right)\right]\notag\\
\le& \left(\E \left[\exp\left((1+\e)C_0C(\gamma)\frac p{2t}\int_{\R^{d+1}} \left(\int_0^t k_{A,a}(t^{-1}(s-u))K_{B,b}(x-X_s) ds\right)^2 dudx\right)\right]\right)^{1/p}\notag\\
&\times \left(\E \left[\exp\left((1+\frac1\e)C_0C(\gamma)\frac q{2t}\int_{\R^{d+1}} \left(\int_0^t  k_{A,a}(t^{-1}(s-u))\tilde K_{B,b}(x-X_s) ds\right)^2 dudx\right)\right]\right)^{1/q},\notag
\end{align}
where $\tilde K_{B,b}(x)=K(x)-K_{B,b}(x)$. Noting that $ k_{A,a}(u)$ is supported on $[-2A, 2A]$ and  is uniformly bounded (say, by $L$), we have 
\begin{align*}
&\E \left[\exp\left((1+\frac1\e)C_0C(\gamma)\frac q{2t}\int_{\R^{d+1}} \left(\int_0^t k_{A,a}(t^{-1}(s-u))\tilde K_{B,b}(x-X_s) ds\right)^2 dudx\right)\right]\\
\le & \E \left[\exp\left((1+\frac1\e)C_0C(\gamma)L^2(4A+2)\frac q{2t}\int_{\R^{d}} \left(\int_0^t \tilde K_{B,b}(x-X_s) ds\right)^2 dx\right)\right].
\end{align*}
Using that $\frac{(a+b)^2}{t+s}\le \frac{a^2}{t}+\frac{b^2}{s}$, we have
\begin{align*}
&\frac{1}{t+s}\int_{\R^d}\left(\int_0^{t+s} \tilde K_{B,b}(x-X_s) ds\right)^2dx\\
\le& \frac{1}{t}\int_{\R^d}\left(\int_0^{t} \tilde K_{B,b}(x-X_s) ds\right)^2dx+\frac{1}{s}\int_{\R^d}\left(\int_t^{t+s} \tilde K_{B,b}(x-X_s) ds\right)^2dx\\
=& \frac{1}{t}\int_{\R^d}\left(\int_0^{t} \tilde K_{B,b}(x-X_s) ds\right)^2dx+\frac{1}{s}\int_{\R^d}\left(\int_0^{s} \tilde K_{B,b}(x-(X_{t+s}-X_t)) ds\right)^2dx,
\end{align*}
where the last equality follows from a change of variable for $s$ and the fact that the Lebesgue measure on $\R^d$ is invariant under the translation $x\to x+X_t$. 
Hence, by the independent and stationary properties of the increments of L\'evy processes, we have 
\begin{align*}
&\E \left[\exp\left(\frac C{t+s}\int_{\R^{d}} \left(\int_0^{t+s} \tilde K_{B,b}(x-X_s) ds\right)^2 dx\right)\right]\\
\le&\E \left[\exp\left(\frac Ct\int_{\R^{d}} \left(\int_0^t \tilde K_{B,b}(x-X_s) ds\right)^2 dx\right)\right] \E \left[\exp\left(\frac Cs\int_{\R^{d}} \left(\int_0^s\tilde K_{B,b}(x-X_s) ds\right)^2 dx\right)\right].
\end{align*}
Therefore, 
\begin{align}
&\limsup_{t\to\infty}\frac1t\log\E \left[\exp\left(\frac C{t+s}\int_{\R^{d}} \left(\int_0^{t+s} \tilde K_{B,b}(x-X_s) ds\right)^2 dx\right)\right]\nonumber\\
\le&\limsup_{t\to\infty}\frac1t\log\left(\E \left[\exp\left(C\int_{\R^{d}} \left(\int_0^1 \tilde K_{B,b}(x-X_s) ds\right)^2 dx\right)\right]\right)^t\nonumber\\
=&\log\E \left[\exp\left(C\int_{\R^{d}} \left(\int_0^1 \tilde K_{B,b}(x-X_s) ds\right)^2 dx\right)\right]\,. 
\label{e.6.10-add}
\end{align}
By Theorem \ref{thm-exponential} we have by  Dalang's condition \eqref{Dalang} 
\begin{align*}
&\E \left[\exp\left(\theta C(\gamma)\int_{\R^{d}} \left(\int_0^1 K(x-X_s) ds\right)^2 dx\right)\right]\\
& =\E \left[\exp\left(\theta \int_0^1\int_0^1\gamma(X_r-X_s) drds \right)\right]<\infty
\end{align*}
for any $\theta >0$.  
Now letting    $B\to\infty$ and $b\to0$, by  the dominated convergence theorem 
we see that  the term on the right-hand side of \eqref{e.6.10-add} goes to 
$0$.

Now combining all the inequalities after \eqref{eq6.9}, noting that we can choose $\e$ arbitrarily close to 0, and $p$ arbitrarily close to $1$, we have that \eqref{eq6.9} can be reduced to 
\begin{align*}
&\limsup_{t\to \infty}\frac1t\E \left[\exp\left(C_0C(\gamma)\frac 1{2t}\int_{\R^{d+1}} \left(\int_0^t k_{A,a}(t^{-1}(s-u))K_{B,b}(x-X_s) ds\right)^2 dudx\right)\right]\notag\\
&\le \mathbf M(\alpha, \beta_0, d, \gamma). 
\end{align*}
{\bf Step 4.}  Summarizing the arguments in Step 2 and Step 3, we see that
 to obtain the upper bound in Theorem \ref{pa}, it suffices to show 
\begin{align}
&\limsup_{t\to \infty}\frac1t\E \left[\exp\left(\frac \theta {2t}C_0C(\gamma)\int_{\R^{d+1}} \left[\int_0^t k_{A,a}(t^{-1}(s-u))K_{B,b}(x-X_s) ds\right]^2 dudx\right)\right]\notag\\
&\le\theta^{\frac{\alpha}{\alpha-\beta}}\mathbf  M(\alpha, \beta_0, d, \gamma). \label{eq6.10}
\end{align}
In this final step, we will  prove the above inequality.  Fix positive constants $A, a, B, b$ and choose arbitrarily $M>2\max\{A,B\}$.
\begin{align}
&\int_{\R^{d+1}}\bigg[\int_0^t k_{A,a}(u-t^{-1}s)K_{B,b}
(x-X_s)ds\bigg]^2dudx\notag\\
&=\sum_{k\in\Z}\sum_{z\in\Z^d}\int_{[0, M]^{d+1}}\bigg[\int_0^t
k_{A, a}(Mk+u-t^{-1}s)K_{B,b}(Mz+x-X_s)ds\bigg]^2dudx\cr
&\le\int_{[0, M]^{d+1}}\bigg[\sum_{j\in\Z}\sum_{z\in\Z^d}\int_0^t
k_{A, a}(Mj+u-t^{-1}s)K_{B,b}(Mz+x-X_s)ds\bigg]^2dudx\notag\\
&=\int_{[0,M]^{d+1}}\bigg[\int_0^t\widetilde{k}_M(u-t^{-1}s)
\widetilde{K}_M(x-X_s)ds\bigg]^2dudx\,,  \label{eq6.11}
\end{align} 
where
\begin{align}\label{u-14}
\widetilde{k}_M(u)=\sum_{j\in\Z}{
k_{A,a}}(Mj+u)\hskip.1in\hbox{and} \hskip.1in
\widetilde{K}_M(x)=\sum_{z\in\Z^d}{ K_{B,b}}(Mz+x)\,
\end{align}
are $M$-periodic functions.  Note that the summations in \eqref{u-14} are well-defined, since the supports of $k_{A,a}(\cdot)$ and $K_{B,b}(\cdot)$ are bounded domains. 
The process
\begin{equation}  \label{eq6.13}
\phi_t(u,x):=\frac1{t}\int_0^t\widetilde k_M(u-t^{-1}s)\widetilde{K}_M(x-X_s)ds\,,
\hskip.2in (u, x)\in [0, M]^{d+1},
\end{equation}
can be considered as a  process taking  values in the Hilbert space
${L}^2([0, M]^{d+1})$ with the norm denoted by $\|\cdot\|$.   Since $\widetilde{k}_M$ and $\widetilde{K}_M$ are
bounded, smooth functions with bounded derivatives, there is a constant $C>0$, such that
$$
\|\phi_t(\cdot, \cdot)\|\le C\hskip.1in \hbox{and}\hskip.1in
\|\phi_t(\cdot +u_1, \hskip.05in\cdot +x_1)
-\phi_t(\cdot +u_2, \hskip.05in \cdot+x_2)\|\le C\vert (u_1, x_1)-(u_2, x_2)\vert
$$
for all $t$ and $(u_1, x_1), (u_2, x_2)\in [0,M]^{d+1}$.
Let $\mathbb K$ be the closure of the following set in 
$L^2([0, M]^{d+1})$: 
$$
\begin{aligned}
\Big\{f\in {L}{ ^2}([0, M]^{d+1}): &\hskip.1in \|f\|\le C
\hskip.05in \hbox{and}\hskip.05in \|f(\cdot +u_1, \hskip.05in\cdot
+x_1)
-f(\cdot +u_2, \hskip.05in \cdot+x_2)\|\\
&\le C\vert (u_1, x_1)-(u_2, x_2)\vert \hskip.05in
\hbox{for}\hskip.05in (u_1, x_1), (u_2, x_2)\in [0,M]^{d+1}\Big\}.
\end{aligned}
$$
 Then, $\phi_t$ defined in \eqref{eq6.13} belongs to  $\mathbb K$, and it follows from \cite[Theorem IV8.21]{DS} that  $\mathbb K$ is compact in ${L}{ ^2}([0, M]^{d+1})$.

Let $\delta>0$ be fixed. For any $g\in \mathbb K$, noting that  the set of  bounded and continuous functions are dense in ${L}{ ^2}([0, M]^{d+1})$, 
the Hahn-Banach theorem (\cite{Yosida}) implies that there is a bounded and
continuous function $f\in {L}{ ^2}([0, M]^{d+1})$ such that $\|g\|^2<-\|f\|^2 +2\langle f, g\rangle+\delta $. By the finite cover theorem for compact sets, one can find finitely many bounded and continuous functions $f_1,\cdots,f_m$ such that $\|g\|^2<\delta +\max_{1\le i\le m}\{-\|f_j\|^2+2\langle f_i,
g\rangle\}$ for all $g\in \mathbb K$. In particular, we have, noting that $\phi_t\in \mathbb K$,
\begin{align*}
\E \left[e^{\frac12\theta t\|\phi_t\|^2} \right]\le e^{\frac12\delta\theta t}
\sum_{i=1}^me^{-\frac12\theta t\|f_i\|^2}\E\left[ e^{\theta t\langle f_i, \phi_t \rangle}\right]\,.
\end{align*}
Therefore,
\begin{align} \label{eq6.14}
\limsup_{t\to\infty}\frac1t\log \E\left[e^{\frac12\theta t\|\phi_t\|^2}  \right] \le \frac12\delta+\max_{1\le i\le m}\left\{-\frac12\theta \|f_i\|^2+\limsup_{t\to\infty}\frac1t \log\E\left[e^{\theta t \langle f_i, \phi_t\rangle}\right] \right\}.
\end{align}
Notice that, for $i=1, \dots, m,$ 
$$
t\langle f_i, \phi_t\rangle=\int_0^t\bigg[\int_{[0, M]^{d+1}}
f_i(u,x)\widetilde{k}_M(u-t^{-1}s)\widetilde{K}_M(x-X_s)dudx\bigg]ds
=\int_0^t\bar{f}_i\Big({s\over t}, X_s\Big)ds\,,
$$
where
$$
\bar{f}_i(s, x)=\int_{[0, M]^{d+1}}
f_i(u,y)\widetilde{k}_M(u-s)\widetilde{K}_M(y-x)dudy\hskip.2in (s, x)
\in [0,1]\times\R^d.
$$
Since $\widetilde K_M$ is a periodic function and $\widetilde K_M(x-X_s)=\widetilde K_M(x-X_s^M),$
we have that 
$$t\langle f_i, \phi_t\rangle
=\int_0^t\bar{f}_i\Big({s\over t}, X_s^M\Big)ds\,.
$$
It is easy to check that $\bar f_i$
satisfies the condition in Proposition \ref{prop-fk}. Hence, 
\begin{align*}\label{u-15'}
 \lim_{t\to\infty}{1\over t}\log\E\left[e^{\theta t \langle f_i,
\phi_t\rangle}\right]  &=\sup_{g\in {\cal
A}_{\alpha, d}^M}\bigg\{\theta \int_0^1\!\!\int_{\mathbb T_M^d}\bar{f}_i(s,x)g^2(s,x)dxds
-\int_0^1\!\!\mathcal E_{\alpha,M} (g(s,\cdot), g(s,\cdot))ds\bigg\},
\end{align*}
where
\begin{equation*} 
\mathcal A^M_{\alpha,d}=\left\{g(s,\cdot)\in L^2(\mathbb T_M^d): \|g(s,\cdot)\|_{\mathbb T_M^d}=1, \forall s\in[0,1] \text{ and } \int_0^1\mathcal E_{\alpha,M} (g(s,\cdot), g(s,\cdot))ds<\infty\right\}.  
\end{equation*}
Notice that 
\begin{align}
 &\int_0^1\!\!\int_{\R^d}\bar{f}_i(s,x)g^2(s,x)dxds\notag\\
=&\int_{[0, M]^{d+1}}f_i(u,y)\bigg[\int_0^1\!\!\int_{\mathbb T_M^d}
\widetilde{k}_M(u-s)\widetilde{K}_M(y-x)g^2(s, x)dxds\bigg]dudy\notag\\
\le & \frac12\|f_i\|^2+\frac12 \int_{[0,M]^d}\int_{\R}\left[\int_0^1\!\!
\int_{\mathbb T_M^d}|u-s\vert^{-\frac{1+\beta_0}2}\widetilde K_{M}(y-x)g^2(s, x)dxds\right]^2dudy. \label{eq6.15}
\end{align}
Since $\delta$ in \eqref{eq6.14} can be arbitrarily small and $M$ in \eqref{eq6.15} can be arbitrarily large,  the desired inequality \eqref{eq6.10} follows from inequalities \eqref{eq6.11} -- \eqref{eq6.15} and Lemma \ref{lemma7.3}. 
\subsection{When $\rho=1$ under the condition \eqref{Dalang}}\label{section6.2}
In this subsection, we consider the Skorohod case, i.e., $\rho=1$, under the condition \eqref{Dalang}, by applying the methodology used in Section \ref{section6.1}. However, under condition \eqref{Dalang}, there will be a technical issue in step 1, since the left-hand side of \eqref{upperbound} is infinity if condition \eqref{Dalang'} is violated. To deal with  this issue, we will first, do step 2 for  $n$-th moments which reduces $|s|^{-\beta_0}$ to a smooth function with compact support, and then, we do step 1 to reduce the $n$-th moment to first moment. 

More precisely, as in Step 1 in Section \ref{section6.1}, when $\rho=1$, \eqref{eq6.1'} is equivalent to 
\begin{align}
&\limsup_{t\to\infty}\frac1t\log\E\left[\exp\left(\frac1{(n-1)t}\sum_{1\le j<k\le n} \int_0^{t}\int_0^{t} \frac{\gamma(X_r^j-X_s^k)}{|t^{-1}(r-s)|^{\beta_0}}drds\right)\right]\notag\\
&\le \mathbf M(\alpha, \beta_0, d, \gamma) \label{upperbound'}
\end{align}
Recall that $k_{A,a}(u)$ is defined in \eqref{kaa}. Let 
\[
\psi_{A,a}(u)=C_0\int_{\R}k_{A,a}(u-v)k_{A,a}(v)dv
\]
 and 
 \[
 \tilde \psi_{A,a}(u)=|u|^{-\beta_0}-\psi_{A,a}(u)\,.
\]
Then, by H\"older's inequality, we have
\begin{align}
&\E\left[\exp\left(\frac1{(n-1)t} \sum_{1\le j<k\le n} \int_0^{t}\int_0^{t} \frac{\gamma(X_r^j-X_s^k)}{|t^{-1}(r-s)|^{\beta_0}}drds\right)\right]\notag\\
\le& \left(\E \left[\exp\left(p\frac{C_0C(\gamma)}{(n-1)t}\sum_{1\le j<k\le n}   \int_0^{t}\int_0^{t} \psi_{A,a}(t^{-1}(r-s))\gamma(X_r^j-X_s^k)drds\right)\right]\right)^{1/p}\notag\\
&\times \left(\E \left[\exp\left(q\frac{C_0C(\gamma)}{(n-1)t}\sum_{1\le j<k\le n}   \int_0^{t}\int_0^{t} \tilde\psi_{A,a}(t^{-1}(r-s))\gamma(X_r^j-X_s^k)drds\right)\right]\right)^{1/q}.\label{eq6.16}
\end{align}
 Therefore, using a similar argument which reduces \eqref{upperbound}  to \eqref{eq6.9}, one can show that to prove \eqref{upperbound'}, it is suffices to prove
\begin{align}\label{eq6.20}
&\limsup_{t\to\infty}\frac1t\log\E\left[\exp\left(\frac1{(n-1)t}\sum_{1\le j<k\le n} \int_0^{t}\int_0^{t} \psi_{A,a}(t^{-1}(r-s))\gamma(X_r^j-X_s^k)drds\right)\right]\notag\\
&\le \mathbf M(\alpha, \beta_0, d, \gamma) \,, 
\end{align}
provided that, for any $\lambda >0$
\begin{align}\label{eq6.20'}
\lim_{\substack{A\to\infty\\ a\to 0}}\limsup_{t\to\infty} \frac1t\log\E \left[\exp\left(\lambda  \int_0^{t}\int_0^{t} \tilde\psi_{A,a}(t^{-1}(r-s))\gamma(X_r^j-X_s^k)drds\right)\right]=0. 
\end{align}
Recalling that $\tilde k_{A,a}(u)=|u|^{-\frac{1+\beta_0}{2}}-k_{A,a}(u),$ 
\begin{align*}
&|u|^{-\beta_0}-\psi_{A,a}(u)=C_0\int_\R |u-v|^{-\frac{1+\beta_0}{2}}|v|^{-\frac{1+\beta_0}{2}}dv-C_0\int_\R k_{A,a}(u-v) k_{A,a}(v) dv\\
&\le C\left(\int_\R\tilde k_{A,a}(u-v)|v|^{-\frac{1+\beta_0}{2}}dv+\int_\R k_{A,a}(u-v) \tilde k_{A,a}(v)dv\right)\\
&\le 2C \int_\R \tilde k_{A,a}(u-v) |v|^{-\frac{1+\beta_0}{2}}dv \\
&\le 2C \left(A^{-\frac{\beta_0-\beta_0'}2}\int|u-v|^{-\frac{\beta_0'+1}{2}}|v|^{-\frac{1+\beta_0}{2}}dv+(2a)^{\frac{\tilde \beta_0-\beta_0}2}\int|u-v|^{-\frac{\tilde \beta_0+1}{2}}|v|^{-\frac{1+\beta_0}{2}}dv\right)
\end{align*}
where $0<\beta_0'<\beta_0<\tilde \beta_0<1$ and the last inequality follows from \eqref{eq6.8}. Hence we have 
\begin{equation}\label{eq6.21}
\tilde \psi_{A,a}(u)= |u|^{-\beta_0}-\psi_{A,a}(u)\le C(\beta_0, \beta', \tilde \beta)\left( A^{-\frac{\beta_0-\beta_0'}2} u^{\frac{\beta_0+\beta_0'}{2}}+(2a)^{\frac{\tilde \beta_0-\beta_0}2} u^{\frac{\beta_0+\tilde \beta_0}{2}}\right).
\end{equation}
Therefore, \eqref{eq6.20'} holds because of  \eqref{eq6.21} and the second half of Proposition \ref{bound-1}, and hence \eqref{upperbound'} now is reduced to \eqref{eq6.20}.  

By a similar argument used in Step 1, in order to show \eqref{upperbound'} that has been reduced to \eqref{eq6.20},  it suffices to prove 
\begin{align}
&\limsup_{t\to\infty}\frac1t\log\E\left[\exp\left(\frac1{2t}\int_0^{t}\int_0^{t} \psi_{A,a}(t^{-1}(r-s))\gamma(X_r-X_s)drds\right)\right]\notag\\
&\le \mathbf M(\alpha, \beta_0, d, \gamma).\label{eq6.17}
\end{align}
  The left-hand side now is finite under condition \eqref{Dalang} since $\psi_{A,a}$ is a bounded function. 
 Noting that \eqref{eq6.17} is identical to \eqref{eq6.9}, we may prove it  in the exact same 
 way as in Step 3 and Step 4 in Subsection \ref{section6.1}. 

\section{Appendix}
First, we will prove  the finiteness of $\mathbf M(\alpha, \beta_0, d, \gamma)$ defined in \eqref{M}. Consider a general non-negative definite (generalized) function $\gamma(x)\in \cS'(\R^d)$.  By the  Bochner-Schwartz Theorem, there exists a tempered measure $\mu$ on $\R^d$ such that $\gamma$ is the Fourier transform of $\mu$ in $\cS'(\bR^d)$, i.e.
$$\int_{\bR^d}\varphi(x)\gamma(x)dt=\int_{\bR^d}\cF \varphi(x)\mu(dx) \quad \mbox{for all} \quad \varphi \in \cS(\bR^d).$$
It follows that for $f,g\in \mathcal S(\R^d)$,
\begin{equation}\label{fourier-int}
\int_{\R^d}\int_{\R^d} \gamma(x-y) f(x) g(y)dxdy= \int_{\R^d} \widehat f(\xi) \overline{\widehat g(\xi)}\mu(d\xi).
\end{equation}


\begin{lemma}\label{lemma-1}
Under the Dalang's condition \eqref{Dalang-add},
$$\sup_{g\in \mathcal F_{\alpha, d}}\left\{\theta \int_{\R^d}\int_{\R^d}\gamma(x-y) g^2(x) g^2(y) dxdy- \int_{\R^d}|\xi|^\alpha|\widehat g(\xi)|^2 d\xi\right\}<\infty,$$
for any $\theta>0$, where $\cF_{\alpha,d}$ is given in \eqref{Falphad}
\end{lemma}
\begin{proof}  It suffices to  consider $g\in \cF_{\alpha,d}\cap \mathcal S(\R^d)$, since $\mathcal S(\R^d)$ is dense in $\F_{\alpha,d}$ endowed with the norm 
\[
\|g\|^2=\left(\int_{\R^d}\int_{\R^d} \gamma(x-y) g^2(x) g^2(y) dxdy\right)^{1/2}+  \int_{\R^d}|\xi|^\alpha|\widehat g(\xi)|^2d\xi\,.
\]
By \eqref{fourier-int} and noting that $\|\cF(g^2)(\cdot)\|_{\infty}\le 1,$  we have 
\begin{align*}
&\int_{\R^d}\int_{\R^d}\gamma(x-y) g^2(x)g^2(y) dxdy=\int_{\R^d} |\cF(g^2)(\xi)|^2 \mu(d\xi)\\
&\qquad \le \mu([|\xi|\le N])+\int_{[|\xi|> N]} |(\widehat g*\widehat g)(\xi)|^2 |\xi|^\alpha \frac{\mu(d\xi)}{|\xi|^\alpha}\\
&\qquad \le \mu([|\xi|\le N])+\left\|(\widehat g*\widehat g)(\cdot)|^2 |\cdot|^\alpha\right\|_{\infty}\int_{[|\xi|> N]}  \frac{\mu(d\xi)}{|\xi|^\alpha}.
\end{align*}
Since  $\alpha\in(0,2]$ we see  $|\xi|^{\alpha/2}\le |\xi-\eta|^{\alpha/2}+|\eta|^{\alpha/2}$ for all $\eta\in \R^d$.  Thus, we have  
\begin{align*}
\bigg|(\widehat g*\widehat g)(\xi)|\xi|^{\alpha/2}\bigg|&\le\int_{\R^d}|\widehat g|(\xi-\eta)|\widehat g|(\eta) \left(|\eta|^{\alpha/2}+|\xi-\eta|^{\alpha/2}\right)d\eta \\
&\le 2\left(|\widehat g|(\cdot) * \left(|\widehat g|(\cdot)|\cdot|^{\alpha/2}\right)\right)(\xi).
\end{align*}
By Young's inequality and Parseval's identity,
\begin{align*}
\Big\||\widehat g|(\cdot) * \left(|\widehat g|(\cdot)|\cdot|^{\alpha/2}\right)\Big\|^2_{\infty}\le \|\widehat g\|_{2}^2\int_{\R^d} |\xi|^{\alpha}|\widehat g(\xi)|^2d\xi= \int_{\R^d} |\xi|^{\alpha}|\widehat g(\xi)|^2d\xi.
\end{align*}
Therefore, for any $\theta>0$,
\begin{align*}
&\theta \int_{\R^d}\int_{\R^d}\gamma(x-y) g^2(x)g^2(y)dxdy -\int_{\R^d} |\xi|^{\alpha}|\widehat g(\xi)|^2d\xi \\
&\le \theta \mu([|\xi|\le N])+ \left(\theta \int_{[|\xi|>N]}\frac{\mu(d\xi)}{|\xi|^\alpha}-1\right) \int_{\R^d} |\xi|^{\alpha}|\widehat g(\xi)|^2d\xi.
\end{align*}
Since $\mu(d\xi)$ is tempered and hence locally integrable, $\mu([|\xi|\le N])$ is finite for any $0<N<\infty$. On the other hand, the Dalang's condition \eqref{Dalang-add} implies that $\lim_{N\to\infty}\int_{[|\xi|>N]}\frac{\mu(d\xi)}{|\xi|^\alpha}=0.$ Therefore, for any $\theta>0,$ one can always find $N$ sufficiently large such that 
$$\theta \int_{\R^d}\int_{\R^d}\gamma(x-y) g^2(x)g^2(y) -\int_{\R^d} |\xi|^{\alpha}|\widehat g(\xi)|^2d\xi\le \theta \mu([|\xi|\le  N])<\infty.$$
This concludes the proof. \hfill 
\end{proof}

\begin{lemma}\label{a-1}
Let $\gamma_0(u),u\in \R$ be a locally integrable function. Then, under the Dalang's condition \eqref{Dalang-add},
\begin{align*}
\sup_{g\in \mathcal A_{\alpha,d}}\Bigg\{\theta \int_0^1\int_0^1 \int_{\R^{2d}}\gamma_0(r-s)\gamma(x-y)g^2(s,x)g^2(r,y)dxdydrds\notag\\
-\int_0^1\int_{\R^d}|\xi|^\alpha|\widehat g(s,\xi)|^2d\xi ds\Bigg\}<\infty,
\end{align*}
for any $\theta>0.$
\end{lemma}
\begin{proof}
The result will  be proven by using a similar argument as that 
 in the proof \cite[Lemma 5.2]{Chen'}. Similar as in Lemma \ref{lemma-1}.  Consider $g\in \cA_{\alpha,d}\cap \mathcal S(\R^{d+1})$, and extend $g(s,x)$   periodically in $s$ from $[0,1]\times \R^d$ to $[0,\infty)\times \R^d$,
still denoted by the same notation $g(s,x)$. Then,   we have 
\begin{align*}
&\int_0^1\int_0^1 \int_{\R^{2d}} \gamma_0(r-s)\gamma_0(x-y) g^2(r,x)g^2(s,y)dxdydrds\\
&=2\int_0^1\int_0^r \int_{\R^{2d}} \gamma_0(r-s)\gamma(x-y) g^2(r,x)g^2(s,y)dxdydrds\\
&=2\int_0^1\gamma_0(r)\int_0^{1-r} \int_{\R^{2d}} \gamma(x-y) g^2(r+s,x)g^2(s,y)dxdyds dr\\
&\le 2 \int_0^1|\gamma_0(r)|\int_0^{1} \int_{\R^{2d}} \gamma(x-y) g^2(r+s,x)g^2(s,y)dxdyds dr.
\end{align*}
By \eqref{fourier-int},  we can write 
\begin{align*}
&\int_{\R^{2d}} \gamma(x-y) g^2(r+s,x)g^2(s,y)dxdy=\int_{\R^d}\left(\cF g^2(r+s,\cdot)\right)(\xi)\overline{\left(\cF g^2(s,\cdot)\right)(\xi)}\mu(d\xi)\\
&\le \left(\int_{\R^d}\left|\left(\cF g^2(r+s,\cdot)\right)(\xi)\right|^2\mu(d\xi)\right)^{1/2}\left(\int_{\R^d}\left|\left(\cF g^2(s,\cdot)\right)(\xi)\right|^2\mu(d\xi)\right)^{1/2}\\
&=\left(\int_{\R^{2d}} \gamma(x-y) g^2(r+s,x)g^2(r+s,y)dxdy\right)^{1/2}\left( \int_{\R^{2d}} \gamma(x-y) g^2(s,x)g^2(s,y)dxdy\right)^{1/2}.
\end{align*}
Noting that $g$ is periodic in time, we see  by H\"older inequality, 
\begin{align*}
\int_0^{1} \int_{\R^{2d}} \gamma(x-y) g^2(r+s,x)g^2(s,y)dxdyds\le \int_0^1 \int_{\R^{2d}}\gamma(x-y) g^2(s,x)g^2(s,y)dxdyds.
\end{align*}
Summarizing   the above computations, we obtain 
\begin{align*}
&\int_0^1\int_0^1 \int_{\R^{2d}} \gamma_0(r-s)\gamma(x-y) g^2(r,x)g^2(s,y)dxdydrds\\
&\le 2\int_0^1|\gamma_0(u)|du \int_0^1 \int_{\R^{2d}}\gamma(x-y) g^2(s,x)g^2(s,y)dxdyds.
\end{align*}
Hence,
\begin{align*}
&\sup_{g\in \mathcal A_{\alpha,d}}\Bigg\{\theta \int_0^1\int_0^1 \int_{\R^{2d}}\gamma_0(r-s)\gamma(x-y)g^2(s,x)g^2(r,y)dxdydrds\notag\\
&~~~~~~~~~~~~-\int_0^1\int_{\R^d}|\xi|^\alpha|\widehat g(s,\xi)|^2d\xi ds\Bigg\}\\
&\le \sup_{g\in \mathcal A_{\alpha,d}}\Bigg\{2\theta \int_0^1|\gamma_0(u)|du \int_0^1 \int_{\R^{2d}}\gamma(x-y) g^2(s,x)g^2(s,y)dxdyds\notag\\
&~~~~~~~~~~~~~~~-\int_0^1\int_{\R^d}|\xi|^\alpha|\widehat g(s,\xi)|^2d\xi ds\Bigg\}\\
&=\int_0^1 \sup_{g\in \mathcal A_{\alpha,d}}\Bigg\{2\theta \int_0^1|\gamma_0(u)|du \int_0^1 \int_{\R^{2d}}\gamma(x-y) g^2(s,x)g^2(s,y)dxdy\notag\\
&~~~~~~~~~~~~~~~-\int_0^1\int_{\R^d}|\xi|^\alpha|\widehat g(s,\xi)|^2d\xi \Bigg\}ds\\
&= \sup_{g\in \mathcal F_{\alpha,d}}\Bigg\{2\theta \int_0^1|\gamma_0(u)|du \int_0^1 \int_{\R^{2d}}\gamma(x-y) g^2(x)g^2(y)dxdy\notag\\
&~~~~~~~~~~~~~~~-\int_0^1\int_{\R^d}|\xi|^\alpha|\widehat g(\xi)|^2d\xi \Bigg\}ds,
\end{align*}
where the variation on the right-hand side is finite by Lemma \ref{lemma-1}. \hfill
\end{proof}

The following lemma was used in the proof of upper bound.
\begin{lemma} \label{lemma7.3}
Let $\widetilde K_M$ be defined by \eqref{u-14}.  Then  
\begin{eqnarray}
&&\limsup_{M\to\infty}\sup_{g\in {\cal A}_{\alpha, d}^M}
\bigg\{\frac12 C_0C(\gamma)\int_{[0,M]^d}\int_{\R}\left[\int_0^1\!\!
\int_{\mathbb T_M^d}|u-s\vert^{-\frac{1+\beta_0}2}\widetilde K_{M}(y-x)g^2(s, x)dxds
\right]^2dudy\nonumber\\
&&\qquad \qquad  -\int_0^1 \mathcal E_{\alpha, M}(g(s,\cdot),g(s,\cdot))ds\bigg\}  \le  \mathbf M(\alpha,\beta_0,d,\gamma)\,. \label{e.7.2-add} 
\end{eqnarray} 
\end{lemma}
\begin{proof} By \cite[Lemma A.1]{hnx}, there exists a positive constant $C_{\alpha, d}$,  depending on $(\alpha, d)$ only, such that 
$$|\xi|^{\alpha}=C_{\alpha, d}\int_{\R^d} \frac{1-\cos(2\pi\xi\cdot y)}{|y|^{d+\alpha}}dy,$$
where $C_{\alpha, d}=\int_{\R^d} \frac{1-\cos(\eta\cdot y)}{|y|^{d+\alpha}}dy$ for any $\eta \in \R^d$ with $|\eta|=2\pi.$ By Lemma \ref{lemma4}, we have 
\begin{equation}\label{e5.7}
\mathcal E_{\alpha} (f,f) =  \frac{C_{\alpha, d}}2\int_{\R^d}\int_{\R^d}  \frac{|f(y)-f(x)|^2}{|y-x|^{d+\alpha}}dydx,
\end{equation}
and for any $M$-periodic function $h$, 
\begin{equation}\label{e5.8}
\mathcal E_{\alpha, M} (h,h) =\frac{C_{\alpha, d}}2\int_{[0, M]^d}\int_{\R^d}  \frac{|h(y)-h(x)|^2}{|y-x|^{d+\alpha}}dydx.
\end{equation}

To prove \eqref{e.7.2-add}, for any fixed $M$-periodic (in space) function $g\in \mathcal A_{\alpha, d}^M$, we shall construct a function $f\in \mathcal A_{\alpha,d}$ such that $f\equiv g$ on $[0,1]\times [M^{1/2}, M-M^{1/2}]$ and the difference between $g$
and $f$ on $[0,1]\times (\R^d \setminus[M^{1/2}, M-M^{1/2}])$ is negligible in some suitable sense as $M$ goes to infinity. 

Denote 
\begin{equation}
E_M:=[0,M]^d\setminus[M^{1/2}, M-M^{1/2}]\,. \label{e.7.4-add} 
\end{equation}
 By Lemma 3.4 in \cite{DV1}, for fixed $s\in[0,1]$, there is an $a(s)\in \R^d$ such that 
$$\int_{E_M} g^2(s,x+a(s)) dx\le  2d M^{-1/2}.$$
 We  assume $a\equiv0$, for otherwise we may replace $g(s,\cdot)$ with $g(s,a(s)+\cdot)$ without changing the value inside $\{\}$ in \eqref{e.7.2-add}. Therefore, without loss of generality, we assume for all $s\in[0,1]$,
\begin{equation}\label{C-0'}
\int_{E_M}  g^2(s,x) dx\le  2d M^{-1/2}.
\end{equation}
Define $\varphi(x)=\phi(x_1)\cdots \phi(x_d),\quad x=(x_1,\cdots, x_d)\in \R^d $, where 
\begin{equation*}
\phi(x)=\begin{cases}
x M^{-1/2}, & 0\le x\le M^{1/2},\\
1, & M^{1/2}\le x\le M- M^{1/2},\\
M^{1/2}-x M^{-\frac12}, & M- M^{1/2}\le x \le M,\\
0, & \text{otherwise}, 
\end{cases}
\end{equation*}
 and let 
$$f(s,x)=  g(s, x)\varphi(x)/\sqrt {G(s)},$$
with  $$ G(s) :=\int_{\R^d} g^2(s,y)\varphi^2(y)dy. $$
 Then, \[
\hbox{$|\phi|\le 1, | \phi'|\le M^{-1/2}$ and hence $|\varphi|\le 1, |\nabla \varphi|\le d^{1/2}M^{-1/2}$.}
\]
  Noting that $$1\ge G(s)=\int_{[0,M]^d} g^2(s,y)\varphi^2(y)dy\ge 1-\int_{E_M} g^2(s,y)dy\ge 1- 2dM^{-1/2},$$ 
we have 
\begin{equation}\label{eq.bm}
0< 1- 2dM^{-1/2}\le b_M:=\inf_{s\in[0,1]} G(s) \le 1. 
\end{equation}

Firstly, we compare the second terms in the variations  on both sides of \eref{e.7.2-add}, i.e., compare $J_1:=\int_0^1 \mathcal E_{\alpha} (f(s,\cdot), f(s,\cdot))ds$ with $J:=\int_0^1 \mathcal E_{\alpha, M} (g(s,\cdot), g(s,\cdot))ds$.  Note that
\begin{align*}
& |g(s,y)\varphi(y)-g(s,x) \varphi(x)| =|(g(s,y)-g(s,x) ) \varphi(y) +g(s,x) (\varphi(y)-\varphi(x))|^2\\
&\le (1+\varepsilon)  |g(s,y)-g(s,x)|^2 \varphi^2(y)+ (1+1/\varepsilon ) g^2(s,x) |\varphi(y)-\varphi(x)|^2, 
\end{align*}
for any $\varepsilon>0.$
Therefore,
\begin{align} \label{e5.10}
&\int_{\R^d}\int_{\R^d} \frac{|g(s,y)\varphi(y)-g(s,x) \varphi(x)|^2}{|y-x|^{d+\alpha}}dydx\notag\\
&\le (1+\varepsilon) \int_{\R^d}\int_{\R^d} \frac{|g(s,y)-g(s,x)|^2 \varphi^2(y)}{|y-x|^{d+\alpha}}dydx\notag\\
& +  (1+1/\varepsilon) \int_{\R^d}\int_{\R^d} \frac{g^2(s,x)|\varphi(y)-\varphi(x)|^2}{|y-x|^{d+\alpha}}dydx. 
\end{align}
Now we bound the above two integrals separately.  
For the first integral, it is easy to verify   by \eqref{e5.7} that  
\begin{equation}\label{e5.11}
\frac{C_{\alpha, d}}{2} \int_{\R^d}\int_{\R^d} \frac{|g(s,y)-g(s,x)|^2 \varphi^2(y)}{|y-x|^{d+\alpha}}dydx\le \mathcal E_{\alpha, M}(g(s,\cdot), g(s,\cdot)).
\end{equation}
 For the second integral,  we have first,  for any $\si\in(0,2),$ 
 \begin{align*}
&g^2(s,x) |\varphi(y)-\varphi(x)|^2\\
&\le g^2(s,x) |\varphi(y)-\varphi(x)|^2 (I_{[0,M]^d} (x)+ I_{[0,M]^d} (y) )\\
&= g^2(s,x) |\varphi(y)-\varphi(x)|^{2-\si}|\varphi(y)-\varphi(x)|^{\si} (I_{[0,M]^d} (x)+ I_{[0,M]^d} (y) )\\
&\le 2^{2-\si} d^{\si/2}M^{-\si /2} g^2(s,x) (I_{[0,M]^d} (x)+ I_{[0,M]^d} (y) )(|y-x|^\si \wedge|y-x|^2),
\end{align*}
Consequently,  we have 
\begin{align}
 &\frac{C_{\alpha, d}}{2} \int_{\R^d}\int_{\R^d} \frac{g^2(s,x)| \varphi(y)-\varphi(x)|^2}{|y-x|^{d+\alpha}}dydx\notag\\
&\le C_{\alpha, d}  2^{2-\si} d^{\si/2}M^{-\si /2}\int_{[0,M]^d}\int_{\R^d} \frac{g^2(s,x)(|y-x|^\si \wedge|y-x|^2)}{|y-x|^{d+\alpha}}dydx\notag\\
&\quad +C_{\alpha, d} 2^{2-\si}d^{\si/2} M^{-\si /2}     \int_{\R^d}\int_{[0,M]^d} \frac{g^2(s,x)(|y-x|^\si\wedge|y-x|^2)}{|y-x|^{d+\alpha}}dydx\notag\\
&=C_{\alpha, d} 2^{2-\si}d^{\si/2} M^{-\si /2}   \int_{[0,M]^d} g^2(s,x)dx \int_{\R^d} \frac{|y|^\si\wedge|y|^2}{|y|^{d+\alpha}}dy\notag\\
&\quad +C_{\alpha, d} 2^{2-\si} d^{\si/2}M^{-\si /2}     \int_{[0,M]^d} \int_{\R^d} \frac{g^2(s,x+y)(|x|^\si\wedge|x|^2)}{|x|^{d+\alpha}}dxdy\notag\\
& \le CM^{-\si/2}, \label{e5.12}
\end{align}
for some constant $C$ depending only on $(\alpha, d)$,  where in the last second step, the two integrals are finite for  $\alpha\in(\sigma,2)$.

Combining \eqref{e5.7}, \eqref{e5.8}, \eqref{e5.10}, \eqref{e5.11} and \eqref{e5.12}, and recalling $b_M$  given in \eqref{eq.bm}, we have
\begin{align}
 b_MJ_1&=b_M\int_0^1 \mathcal E_\alpha(f(s,\cdot), f(s,\cdot)) ds\notag\\
&\le \int_0^1 G(s)\mathcal E_\alpha(f(s,\cdot), f(s,\cdot)) ds \notag\\
&\le (1+\varepsilon)\int_0^1 \mathcal E_{\alpha,M} (g(s,\cdot), g(s,\cdot)) ds  +C (1+1/\varepsilon) M^{-\si/2}\notag\\
&=(1+\varepsilon)J+C  (1+1/\varepsilon) M^{-\si/2}. \label{e5.13}
\end{align}

 Secondly,  we estimate the first term inside $\{\}$ in \eqref{e.7.2-add}.  Recall that $K_{B, b}(\cdot)$ is supported on $[-2B, 2B]^d$, hence for any fixed $y\in [0, M]^{d}, K_{B,b}(y-\cdot)$ is supported on $[-2B, M+2B]^d$. Therefore, for $y\in [0,M]^d$,
\begin{align}
&\int_{[0,M]^d}\widetilde K_{M}(y-x)g^2(s, x)dx
=\int_{[0,M]^d}\sum_{z\in\mathbb Z^d} K_{B,b}(y-x+zM) g^2(s, x)dx\notag\\
&=\int_{\R^d} K_{B,b}(y-x) g^2(s, x)dx=\int_{[-2B, M+2B]^d} K_{B,b}(y-x) g^2(s, x)dx \label{eq7.12-add}
\end{align}
where the second equality follows from the $M$-periodicity of $g(s,\cdot)$. 

Denote 
\begin{equation}
\tilde E_M:=[-2B, M+2B]^d\setminus[M^{1/2}, M-M^{1/2}]\,. \label{e.7.4-add'}
\end{equation}
Then there exists a constant $C$ depending only on $d$ such that
\begin{equation}\label{C-0''}
\int_{\tilde E_M}  g^2(s,x) dx\le  C M^{-1/2}, \, \forall s\in[0,1].
\end{equation}
This is because of \eqref{C-0'}, the periodicity of $g(s,\cdot)$ and the fact that there is a partition of $[-2B,M+2B]^d\setminus [0, M]^d$ 
such that the number of parts in the partition is finite depending only on $d$ and  each part from this partition can be shifted by $zM$ for some $z\in\Z^d$ to become a subset of
$[0, M]^d\setminus [2B, M-2B]^d\subset [0, M]^d\setminus [\sqrt{M}, M-\sqrt{M}]^d$ when $M>4B^2$.

Notice that
\[
g^2(s,x)=
G(s)f^2(s,x), \quad \forall \ x\in [M^{1/2}, M-M^{1/2}] = [-2B, M+2B]^d\setminus \tilde E_M \,,
\]
where $\tilde E_M$ is defined by \eqref{e.7.4-add'}.  We can bound the integral in
\eqref{e.7.2-add} as follows, noting \eqref{eq7.12-add},
\begin{align}\label{C-2}
 I:=&  \int_{[0,M]^d}\int_{\R}\left[\int_0^1\!\!
\int_{[0,M]^d}|u-s\vert^{-\frac{1+\beta_0}2}\widetilde K_{M}(y-x)g^2(s, x)dxds\right]^2dudy\notag\\
=&  \int_{[0,M]^d}\int_{\R}\left[\int_0^1\!\!
\int_{[-2B, M+2B]^d}|u-s\vert^{-\frac{1+\beta_0}2} K_{B,b}(y-x)g^2(s, x)dxds\right]^2dudy\notag\\
\le & (1+\varepsilon)\int_{[0,M]^d}\int_{\R}\left[\int_0^1\!\!
\int_{[-2B, M+2B]^d\setminus \tilde E_M}|u-s\vert^{-\frac{1+\beta_0}2}K_{B,b}(y-x) g^2(s, x)dxds\right]^2dudy\notag\\
&+(1+1/\e)\int_{[0,M]^d}\int_{\R}\left[\int_0^1\!\!
\int_{\tilde E_M}|u-s\vert^{-\frac{1+\beta_0}2}K_{B,b}(y-x)  g^2(s, x)dxds\right]^2dudy\notag\\
\le &(1+\e)\max_{s\in [0,1]}G(s)\int_{[0,M]^d}\int_{\R}\left[\int_0^1\!\!
\int_{\R^d}|u-s\vert^{-\frac{1+\beta_0}2}K_{B,b}(y-x)f^2(s,x)dxds\right]^2dudy\notag\\
&+(1+1/\e)\int_{[0,M]^d}\int_{\R}\left[\int_0^1\!\!
\int_{\tilde E_M}|u-s\vert^{-\frac{1+\beta_0}2}K_{B,b}(y-x)   g^2(s, x)dxds\right]^2dudy\notag\\
\le & (1+\e)\left(C_0C(\gamma)\right)^{-1/2}I_1
 +(1+1/\e) I_2\,,
\end{align}

where 
\begin{eqnarray*}
I_1 &:=&  \int_0^1\int_0^1\int_{\R^d\times\R^d}\frac{\gamma(x-y)}{|r-s|^{\beta_0}}f^2(s,x)f^2(r,y)dxdydrds \\
I_2 &:=&\int_{[0,M]^d}\int_{\R}\left[\int_0^1\!\!
\int_{\tilde E_M}|u-s\vert^{-\frac{1+\beta_0}2}K_{B,b}(y-x)   g^2(s, x)dxds\right]^2dudy\,. 
\end{eqnarray*}
We consider $I_2$.   Note  that the function $K_{B,b}(\cdot)$ is uniformly bounded, say, by $D$.
Then we have 
\begin{align}\label{C-3}
I_2
=& C_0^{-1}\int_{[0,M]^d}\int_0^1\int_0^1 |r-s|^{-\beta_0}  dr dsdy \int_{\tilde E_M} K_{B,b}(y-x_1)g^2(s,x_1) dx_1\notag\\
&\int_{\tilde E_M}  K_{B,b}(y-x_2) g^2(r,x_2)dx_2 \notag\\
\le & C_0^{-1}\int_{[0,M]^d}\int_0^1\int_0^1 |r-s|^{-\beta_0}  dr dsdy\int_{\tilde E_M} K_{B,b}(y-x_1)g^2(s,x_1) dx_1\int_{\tilde E_M}Dg^2(r,x_2)dx_2 \notag\\
\le  &CC_0^{-1} M^{-1/2} \int_{\R^d} K_{B,b}(y)dy\int_0^1\int_0^1 |r-s|^{-\beta_0}  dr ds\int_{\tilde E_M} g^2(s,x_1) dx_1\notag\\
\le  &CC_0^{-1} M^{-1/2}\int_{\R^d} K_{B,b}(y)dy(1-\beta_0)^{-1}\int_0^1\left[ s^{1-\beta_0}+(1-s)^{1-\beta_0} \right]\int_{\tilde E_M} g^2(s,x_1) dx_1 ds\notag\\
\le &CC_0^{-1} M^{-1/2}\int_{\R^d} K_{B,b}(y)dy(1-\beta_0)^{-1}(2-\beta_0)^{-1}\int_0^1\int_{\tilde E_M} g^2(s,x_1) dx_1 ds\notag\\
\le &2CC_0^{-1} M^{-1/2}\int_{\R^d} K_{B,b}(y)dy(1-\beta_0)^{-1}(2-\beta_0)^{-1}(2dM^{-1/2})\notag\\
=&C\Big(K_{B,b}(\cdot),d,\beta_0\Big)M^{-1},
\end{align}
where the third step and the last second step follow from (\ref{C-0''}).

Finally, combing (\ref{e5.13}), (\ref{C-2}) and (\ref{C-3}), we can bound the quantity inside
$\{\ \}$ in \eqref{e.7.2-add} as follows   (recall that  $J$ and $J_1$ are defined by \eqref{e5.13} and $b_M$ is given in \eqref{eq.bm})
\begin{align*}
 \frac12 C_0C(\gamma)I-J  
\le &\frac{1+\e}2 I_1 +C(1+1/\e)M^{-1} -\frac{b_M}{1+\varepsilon} J_1+C \frac{1+1/\varepsilon}{1+\varepsilon} M^{-\si/2}\\
\le  &\frac{b_M}{1+\varepsilon}\Bigg\{\frac{(1+\varepsilon)^2}{2b_M}\int_0^1\int_0^1\int_{\R^d\times\R^d}\frac{\gamma(x-y)}{|r-s|^{\beta_0}}f^2(s,x)f^2(r,y)dxdydrds\\
&- \int_0^1 \mathcal E_{\alpha}(f(s,\cdot), f(s,\cdot))ds  \Bigg\} +C(1+1/\e)M^{-1}+C\frac{1+1/\varepsilon}{1+\varepsilon}  M^{-\si /2} .
\end{align*}
Therefore,
\begin{align*}
&\sup_{g\in\mathcal A_{\alpha,d}^M}\left\{ \frac12 C_0C(\gamma) I-J\right\}\\
\le& \frac{b_M}{1+\varepsilon}\sup_{f\in\mathcal A_{\alpha,d}}\Bigg\{\frac{(1+\varepsilon)^2}{2b_M} \int_0^1\int_0^1\int_{\R^d\times\R^d}\frac{\gamma(x-y)}{|r-s|^{\beta_0}}f^2(s,x)f^2(r,y)dxdydrds\\
&-  \int_0^1\mathcal E_\alpha(f(s,\cdot), f(s,\cdot))ds\Bigg\}+C(1+1/\e)M^{-1}+C\frac{1+1/\varepsilon}{1+\varepsilon} M^{-\si/2}\\
= &  \frac{b_M}{1+\varepsilon}  \left( \frac{(1+\varepsilon)^2}{b_M} \right)^{\frac{\alpha}{\alpha-\beta}}\sup_{f\in\mathcal A_{\alpha,d}}\Bigg\{ \frac12\int_0^1\int_0^1\int_{\R^d\times\R^d}\frac{\gamma(x-y)}{|r-s|^{\beta_0}}f^2(s,x)f^2(r,y)dxdydrds\\
&-   \int_0^1\mathcal E_\alpha(f(s,\cdot), f(s,\cdot))ds\Bigg\}+C(1+1/\e)M^{-1}+C\frac{1+1/\varepsilon}{1+\varepsilon} M^{-\si /2},
\end{align*}
where the last step follows from \eqref{mscaling}. 
Noting that  $\lim_{M\to\infty} b_M=1$, we have, by choosing $\varepsilon$ arbitrarily small,
\begin{align*}
&\lim_{M\to \infty}\sup_{g\in\mathcal A_{\alpha,d}^M}\Bigg\{  \frac12 C_0C(\gamma)I-J \Bigg\}\le    \sup_{f\in\mathcal A_{\alpha,d}}\Bigg\{\frac12 \int_0^1\int_0^1\int_{\R^d\times\R^d}\frac{\gamma(x-y)}{|r-s|^{\beta_0}}\\
&\qquad \qquad \qquad\qquad   f^2(s,x)f^2(r,y)dxdydrds-   \int_0^1\mathcal E_\alpha(f(s,\cdot), f(s,\cdot))ds\Bigg\}.
\end{align*}
Hence \eqref{e.7.2-add} is proved, provided  $\alpha\in (\sigma,2)$. Note that $\sigma\in(0,2)$ is arbitrary,  therefore \eqref{e.7.2-add} holds for $\alpha\in(0,2).$ 

The proof is concluded, noting that for the case $\alpha=2$, \eqref{e.7.2-add} can be proved in a similar way as in \cite[Lemma A.3]{CHSX}.  
\end{proof}

 \begin{lemma}\label{lemma4} Let $f\in L^2(\R^d)$ and $h\in L^2(\mathbb T_M^d)$. Then,
\begin{align}\label{e.lemma4-1}
&2\int_{\R^d} \Big(1-\cos(2\pi \xi\cdot y)\Big) |\widehat f(\xi)|^2 d\xi=\int_{\R^d}|f(x+y)-f(x)|^2 dx,
\end{align}
and
\begin{align}\label{e.lemma4-2}
&\frac{2}{M^d}\sum_{k\in \mathbb Z^d} \Big(1-\cos(2\pi k\cdot y)\Big) |\widehat h(k)|^2 =\int_{[0,M]^d}|h(x+My)-h(x)|^2 dx.
\end{align}
\end{lemma}
\begin{proof}
We will prove \eqref{e.lemma4-2} only, and \eqref{e.lemma4-1} can be proved in the same spirit. 
Noting that $1-\cos(2\pi k\cdot y)=2\sin^2(\pi k\cdot y)$, we have
\begin{align*}
&\frac{2}{M^d}\sum_{k\in \mathbb Z^d} \Big(1-\cos(2\pi k\cdot y)\Big) |\widehat h(k)|^2 =\frac{1}{M^d}\sum_{k\in \mathbb Z^d} |2\sin(\pi k\cdot y)\widehat h(k)|^2\\
=&\frac{1}{M^d}\sum_{k\in \mathbb Z^d} \left|\left(e^{i\pi k\cdot y}-e^{-i\pi k\cdot y}\right)\widehat h(k)\right|^2=\int_{[0,M]^d} \left |h(x+\frac {My}2)-h(x-\frac {My}2)\right|^2 dx.
\end{align*}
The last equality holds because of the Parseval's identity 
$$\frac1{M^d} \sum_{k\in \mathbb Z^d} |\widehat g(k)|^2 =\int_{[0,M]^d} \left |g(x)\right|^2 dx, $$
and the fact that for any $a\in \R^d$ and any $M$-periodic function $g$,
\begin{align*}
& \widehat{g(\cdot +a)}(k)=\int_{[0,M]^d} e^{-2\pi i k\cdot y /M} g(y+a) dy=\int_{[0,M]^d} e^{-2\pi i k\cdot (y+a)/M} g(y+a) dy \, e^{2\pi i k\cdot a/M }\\
&=\int_{[0,M]^d} e^{-2\pi i k\cdot y/M} g(y) dy \, e^{2\pi i k\cdot a/M } =  \widehat g(k) e^{2\pi i k\cdot a/M},
\end{align*}
where the  third equality holds because $e^{-2\pi i k\cdot y/M} g(y)$ is an $M$-periodic function in $y$.
\end{proof}

\end{document}